\documentclass[letterpaper, 11pt,  reqno]{amsart}
\usepackage[margin=1.2in,marginparwidth=1.5cm, marginparsep=0.5cm]{geometry}

\usepackage{amsmath,amssymb,amscd,amsthm,amsxtra, esint, xcolor,mathtools, mathrsfs}

\usepackage[implicit=true,pdfborder={0 0 0}]{hyperref}
\hypersetup{hidelinks}

\allowdisplaybreaks[2]

\usepackage{color}

\newtheorem{theorem}{Theorem} [section]

\newtheorem{lemma}[theorem]{Lemma}
\newtheorem{proposition}[theorem]{Proposition}
\newtheorem{remark}[theorem]{Remark}
\newtheorem{example}[theorem]{Example}

\newtheorem{definition}[theorem]{Definition}
\newtheorem{corollary}[theorem]{Corollary}

\newcommand{\noi}{\noindent}
\newcommand{\Z}{\mathbb{Z}}
\newcommand{\R}{\mathbb{R}}

\newcommand{\T}{\mathbb{T}}

\newcommand{\bul}{\bullet}

\newcommand{\DL}{\Delta}

\newcommand{\E}{\mathbb{E}}

\newcommand{\F}{\mathcal{F}}

\newcommand{\cN}{\mathcal{N}}

\newcommand{\al}{\alpha}

\newcommand{\dl}{\delta}

\newcommand{\eps}{\varepsilon}

\newcommand{\g}{\gamma}

\newcommand{\s}{\sigma}

\newcommand{\wt}{\widetilde}

\newcommand{\LRA}{\Longrightarrow}

\renewcommand{\O}{\Omega}

\newcommand{\ind}{\mathbf 1}

\newcommand{\PP}{\mathbb{P}}

\def\e{\varepsilon}

\newcommand{\N}{\mathbb{N}}

\newcommand{\Id}{\textup{Id}}
\newcommand{\Var}{\textup{Var}}
\newcommand{\fm}{\mathfrak{m}}

\newcommand{\fq}{\mathfrak{q}_\ast}

\newcommand{\fH}{\mathfrak{H}}

\numberwithin{equation}{section}
\numberwithin{theorem}{section}

\makeatletter
\@namedef{subjclassname@2020}{\textup{2020} Mathematics Subject Classification}
\makeatother

\title[BMD principle]{Breuer-Major-Donsker  invariance principle}

\author[P.~Mansanarez, G.~Poly, and G.~Zheng]{Paul Mansanarez, Guillaume Poly,  and Guangqu Zheng}

\date{\today}

\address{
Paul Mansanarez\\
Nantes Universit\'e/Universit\'e libre de Bruxelles\\
France/Belgium. 
}

\email{paul.mansanarez@ulb.be}

\address{
Guillaume Poly\\
Nantes Universit\'e \\
France.
}

\email{guillaume.poly@univ-nantes.fr}

\address{
Guangqu Zheng\\
Boston University\\
USA.
 }

\email{gzheng90@bu.edu}

\subjclass[2020]{60F17, 60G46, 60G15.}

\keywords{Donsker's invariance principle;
Breuer-Major theorem; tightness; martingale functional CLT;
hypercontractivity; non-determinism; Kolmogorov's prediction formula.}

\begin{document}
\baselineskip = 14pt

\begin{abstract}
We prove a Breuer-Major-type Donsker's invariance principle for 
stationary Gaussian sequences 
under the natural {\it finite-variance assumption} on the test function.  
This result, which we call the {\it Breuer-Major-Donsker} principle, 
or simply the {\it BMD principle}, 
removes the additional moment assumption 
imposed in the functional Breuer-Major theorem of Nourdin and Nualart
({\it  Probab. Theory Related Fields}, 2020).

Our method does not rely on the Malliavin-calculus estimates used by Nourdin 
and Nualart, in particular Meyer's inequality.  Instead, it is based on a
predictable--martingale decomposition of the partial-sum process,
which is of independent interest. 
We also make systematic use of  {\it non-determinism},
a central  notion in   Gaussian prediction theory.
In the non-deterministic case, the martingale part is handled by the martingale
functional central limit theorem, while the predictable remainder gains
integrability above order two through Ornstein-Uhlenbeck smoothing.  
In the deterministic case, the martingale part vanishes, and the smoothing mechanism is
no longer available along the full sequence.  
Nevertheless, 
under an additional
mild assumption on the covariance function, a suitable decimation recovers
non-determinism and reduces the proof to the non-deterministic case.
\end{abstract}

\maketitle

\tableofcontents

\section{Introduction}

In the 1950s, M.D.~Donsker first provided a general invariance principle 
for sums of independent random variables.
More precisely, 
Donsker's invariance principle 
identifies Brownian motion as the universal scaling limit of 
centered partial sums  with finite variance
(\cite{Don51, Don52}).
P.~Billingsley derived a principle for stationary and ergodic
martingale differences in his 1968 book (see also \cite{Bill99}),
extending an earlier work \cite{Ros67} by
B.~Ros\'en on central limit theorems  (CLTs) for dependent random variables. 
See, e.g., the book \cite{HH80} by Hall and Heyde for a brief overview
and see also the survey \cite{Whitt07} on martingale CLTs.

 In 1983, P.~Breuer and  P.~Major 
 provided a  CLT  \cite{BM83}, 
 motivated by  the non-central limit theorems of Dobrushin, Major, Rosenblatt, 
 and Taqqu during  1977-1981 (see \cite{Taqqu77,Taqqu79, DM79, Major81,Ros81}). 
 Unlike these works,  Breuer and Major were interested in the asymptotic normality 
 of nonlinear functionals over stationary  Gaussian  fields 
 when the corresponding  correlation function decays fast enough. 
 Let us first fix some notations and state the Breuer-Major theorem. 
 Let 
$
\mathbf{G}:= (G_k)_{k\in\Z}
$
be a centered stationary Gaussian sequence such that
\[
\E[G_kG_\ell]=\rho_G(k-\ell)  
\]
 with $ \rho_G(0)=1$ (i.e., $G_0\sim\cN(0,1)$). 
 Let $\g$ denote the standard Gaussian measure on $\R$ 
 throughout this paper. 
Let $\{ H_k: k\in\N\cup\{0\}\}$ be the Hermite polynomials
defined by 
$H_k(x) = (-1)^k e^{x^2/2} \frac{d^k}{dx^k} (e^{-x^2/2})$.\footnote{The first few Hermite polynomials 
are given by $H_0(x) =1$, $H_1(x)= x$, $H_2(x) = x^2-1$, and $H_{p+1}(x) = xH_p(x) - p H_{p-1}(x)$
for $p\geq 2$. See, e.g., \cite[Section 1.4]{NP12}.
} It is well known that $\{ \frac{H_k}{\sqrt{k!}}: k\in\N\cup\{0\}\}$ is an orthonormal basis of 
the Hilbert space $L^2(\g)$, i.e., any
  $\varphi\in L^2(\g)$ admits an expansion in Hermite polynomials.
  More precisely, we assume that 
  $\varphi$ has {\bf Hermite rank} $d\geq 1$, meaning
\begin{align}\label{her_varphi}
\varphi = \sum_{q\geq d} c_q H_q, 
\end{align}
where $c_d \neq 0$ and $\| \varphi\|_{L^2(\g)}^2 = \sum_{q\geq d} c_q^2 q! < +\infty$. 
Define the process

\noi
\begin{align}\label{def_VN}
V_n(t) 
= \frac{1}{\sqrt{n}} \sum_{k=1}^{\lfloor nt \rfloor} \varphi(G_k)
\text{\quad for $t\in [0,1]$. }
\end{align}

\begin{theorem}[\textsf{Breuer-Major theorem} \cite{BM83}]
\label{thm_BM}
Let $\{G_n\}_{n\in\Z} $ be a  centered stationary Gaussian sequence 
with covariance $\rho_G$ and let $\varphi \in L^2(\g)$ have Hermite rank $d\ge 1$
as in \eqref{her_varphi}.
If  $\rho_G\in\ell^d(\Z)$,
then the finite-dimensional distributions of  $V_n$
converge to those of $\s W$, where $W$ is standard Brownian motion and 
\begin{align}\label{def_s}
 \s^2 
 := \sum_{q=d}^\infty q! c_q^2 \sum_{k\in\Z} \rho_G(k)^q\in[0,\infty). 
\end{align}
Here,    $\sigma^2\in[0,\infty)$  is part of the conclusion.
\end{theorem}

Theorem \ref{thm_BM} was first proved by Breuer and Major in \cite{BM83}
using   the method of moments, while 
we refer interested readers to \cite[Chapter 7]{NP12} 
for a proof based on the fourth moment theorems
of Nualart, Peccati, and Tudor. 
Note that any CLT-type result involving Hermite ranks and summability of covariance functions
 is  called a
``Breuer-Major theorem'' nowadays, 
in honor of the seminal paper \cite{BM83}; see, e.g., the introduction in  \cite{NPP11}
for a brief history of Theorem \ref{thm_BM}. 
Since its appearance, the Breuer-Major theorem
has become a standard tool for the asymptotic analysis of nonlinear functionals
of Gaussian sequences and fields. Its applications include 
empirical processes under Gaussian subordination
\cite{DT89,CM96}, power variations and statistical
estimation for fractional Gaussian models \cite{Coe01,NNT10},
level-crossing and excursion functionals \cite{KL01,Kratz06}, 
wavelet estimation of  memory parameters
 \cite{BT10, KLLM13}, 
 spatial  averages and continuous Gaussian-field limits \cite{CNN20,NZ20}, 
random corrector problem in homogenization \cite{NZ20_osc},
 and geometric statistics of Gaussian random
fields \cite{EL16,MN24}, to name a few.

%
%
%

It is not difficult to see that Theorem \ref{thm_BM} generalizes the classical CLT for
sums of i.i.d. random variables with finite variance.\footnote{Indeed, 
suppose $\{ X_j\}_{j\in\Z}$
is an i.i.d. sequence, then one can write $X_j = \varphi(G_j)$
with $\varphi = F_X^{-1}\circ \Phi$ and $\{G_j\}$ i.i.d. standard normal, 
where $F_X^{-1}$ is the generalized quantile function for the distribution of $X_0$
and $\Phi$ is the standard normal cumulative function.
} 
 Then, the functional version of Theorem \ref{thm_BM}
generalizes   Donsker's invariance principle
to the dependent case. 
The tightness problem for the corresponding functional CLT remained open, at
this level of generality, for more than three decades.  Nourdin and Nualart
settled it under the additional moment assumption \eqref{p>2}, through an elegant
application of Meyer's inequality (\cite{Meyer84}), a fundamental tool in Malliavin calculus;
see \cite[Theorem~1.1]{NN20}.

\begin{theorem}[\cite{NN20}]
\label{thm_NN20}
Under the same setting as Theorem~\ref{thm_BM},
assume additionally that
 \begin{align}\label{p>2}
\text{$\varphi\in L^{p}(\g)$
for some $p>2$.}
\end{align}

\noi
 Then, the process $V_n$ converges in law in $\mathbf{D}([0,1])$
to $\s W$. 
\end{theorem}

Here, $\mathbf{D}([0,1])$
stands for  the space of all real-valued
c\`adl\`ag functions on $[0,1]$, that is, all functions
$x:[0,1]\to\R$, which are right-continuous with  left limits.
We endow $\mathbf{D}([0,1])$ with the Skorokhod topology.
See Section \ref{SEC_22} for more details.

\bigskip

 Note that the moment assumption \eqref{p>2}
 excludes many test functions in $L^2(\g)$, 
 for example, 
  $\varphi(x) = (2\pi)^{1/4} \exp(\frac{x^2}{4} - |x|)$.
Our work is motivated by the important contribution \cite{NN20} to functional CLTs,
and our goal is to remove the additional assumption \eqref{p>2} in Theorem \ref{thm_NN20}, that is, 
we aim at establishing the Breuer-Major-type  Donsker's  invariance principle 
under the  natural  finite-variance assumption. In the following, we will call it 
the {\bf\textsf{BMD principle}}.

\subsection{Main results: \textsf{BMD principle}}
Let us first recall an important notion (non-determinism) 
in Gaussian prediction theory \cite{Bing22}. 
Let
$
\F_k:=\sigma\{G_j:j\le k\}$
 for $k\in\Z$.
 For a stationary Gaussian sequence $\mathbf{G}$,
 we say it is  {\bf non-deterministic} 
in the one-step prediction sense
if 

\noi
\begin{align} \label{def_nu1}
\nu^2:=\E\big[\big(G_0-\E[G_0\mid \F_{-1}]\big)^2\big] > 0,
\end{align}

\noi
which is equivalent to  $G_0\notin L^2(\Omega, \F_{-1}, \PP)$.
The quantity $\nu^2$, called the one-step prediction error, 
measures the variance of the new innovation not predicted from the past.
The case $\nu^2 > 0$ is the 
{\it\textsf{non-deterministic case}}
and $\nu^2 = 0$ is called  the {\it\textsf{deterministic case}}.\footnote{The deterministic case
 should be understood as a situation, in which
the whole randomness of the process is already present in the remote past.
Equivalently, the filtration generated by the process carries no genuine
innovations: adding one more observation does not enlarge the information.
In particular, the present variable $G_0$ is encoded not only in the
immediate past, but in every arbitrarily remote past sigma-field
$\F_j$, $j\leq -1$. For example,  given $A, B$ two independent standard normals
and $\theta\in\R$, 
the sequence 
$G_k = A \cos(k\theta) + B \sin(k\theta)$, $k\in\Z$
is then stationary and deterministic. }
%
By Herglotz's theorem (\cite[Theorem 4.3.1]{BD91}),
the correlation function $\rho_G$ admits a spectral 
representation:

\noi
\begin{align*} 
\rho_G(k) = \frac{1}{2\pi}\int_{-\pi}^\pi e^{-i k \xi} \mu(d\xi),
\end{align*}

\noi
where $\mu$ is known as the spectral measure associated to $\mathbf{G}$. 
If $\mu$ admits a density $\fm$, 
then the classical Kolmogorov's prediction formula
connects non-determinism directly to the spectral density $\fm$
 through

\noi
\begin{align} \label{def_nu2}
\nu^2 = \exp\bigg( \frac{1}{2\pi} \int_{-\pi}^{\pi} \log \fm(\lambda) d\lambda\bigg).
\end{align}
See \cite[Theorem 5.8.1]{BD91}.
\noi
In the i.i.d. case where $\fm \equiv 1$, we have $\nu^2 = 1$.
For example, if $\fm = 0$ on a set of positive measure, 
then  $\nu^2 = 0$ and thus $\mathbf{G}$ is deterministic in 
the one-step 
prediction sense. 
See also Examples \ref{exam1},  \ref{exam2}, and \ref{not18}
in Section \ref{SEC_32}.

\bigskip

Now we are ready to state our first main result, whose proof is given in Section \ref{SEC_31}.

\begin{theorem}[\textsf{Non-deterministic case}]  \label{thm_BMD1}
Let   $\nu^2 > 0$ 
as in \eqref{def_nu1}.  Let the assumptions of Theorem \ref{thm_BM} hold. 
Then, as $n\to+\infty$,
 $V_n$ converges in law to $\s W$ in $\mathbf{D}([0,1])$,
 denoted by $V_n\LRA \s W$.

\end{theorem}

The key ingredient  in proving Theorem \ref{thm_BMD1} 
is that the one-step prediction of the Gaussian sequence 
naturally splits the summands into two pieces.  
Define
\[
\DL_k:=\E[\varphi(G_k)\mid \F_{k-1}],
\quad
\xi_k:=\varphi(G_k)-\DL_k,
\]
and for \(t\in[0,1]\),  

\noi
\begin{align}\label{def_MN}
M_n(t):=\frac1{\sqrt n}\sum_{k=1}^{\lfloor nt\rfloor}\xi_k,
\quad
{\rm and}
\quad
N_n(t):=\frac1{\sqrt n}\sum_{k=1}^{\lfloor nt\rfloor}\DL_k.
\end{align}

\noi
Then, 
$(\xi_k,\F_k)_{k\in\Z}$ is a stationary martingale difference
sequence,  and hence the tightness of the martingale part  
 falls under the martingale functional CLT
 (Proposition \ref{MCLT}).  
 The predictable part  $N_n$ is smoother: when the original sequence is non-deterministic, 
 conditioning applies a genuine Ornstein-Uhlenbeck semigroup to $\varphi$
 and then  Nelson's hypercontractivity (Proposition \ref{prop_1}) 
 upgrades this predictable part from $L^2(\O)$ to some $L^p(\O)$, $p>2$, 
 without any additional assumption on $\varphi$.  
 This is precisely the amount of integrability needed for the standard tightness criterion 
 (Lemma \ref{lem_NN20}).
 If the sequence $\mathbf{G}$ is deterministic, 
 the martingale part disappears (i.e., $M_n \equiv 0$)
 and   the smoothing mechanism is
no longer available {\it along the full sequence}.
If the sequence  $\mathbf{G}$ admits $q$-decimated subsequences (for some $q\geq 2$)
that are non-deterministic, then
tightness for the original process follows by recombining finitely many decimated subsequences.
For an integer $q\geq 2$
and
 $r\in\{0,\ldots,q-1\}$, we define the decimated Gaussian sequence
\begin{equation}\label{def_GR0}
\text{$G_j^{(r)}:=G_{qj+r},$
 $j\in\Z$}.
\end{equation}
Note that by stationarity of $\mathbf{G}$, the law of $G^{(r)}$ does not depend on $r$.

\medskip

Now we are ready to state our second main result.

\begin{theorem}[\textsf{Deterministic case}]  
\label{thm_BMD2}
{\rm (i)} 
Assume that $\nu^2 = 0$.
Suppose  $\varphi\in L^2(\g)$ has 
   Hermite rank
$d\geq 1$ as in \eqref{her_varphi}.
Assume
$
\rho_G\in\ell^d(\Z)
$
and 
\begin{align}\label{DEC}
\begin{aligned}
&\text{$\{ G_{qj}: j\in\Z\}$ is non-deterministic
for some $q\geq 2$.}
\end{aligned}
\end{align}

\noi
Then,  
$
V_n
\LRA
\s W$
 in $\mathbf{D}([0,1])$,
where $W$ and  $\s^2$ are given  as in 
Theorem \ref{thm_BM}.

\smallskip
\noi
{\rm (ii)} If \begin{align} \label{suff_tight}
\text{$k\in\Z\mapsto \rho_G(\fq k)$ is absolutely summable for some $\fq\geq 2$,}
\end{align}
then  the condition \eqref{DEC} holds for some $q$
that is divisible by $\fq$.

\end{theorem}

We present the proof  of Theorem \ref{thm_BMD2} in Section \ref{SEC_32}.
As one will see in the proof of part (ii),
    under the mild summability condition \eqref{suff_tight} along an arithmetic subsequence, 
   a sufficiently large decimation has a spectral density bounded away from zero, 
   and is therefore non-deterministic by Kolmogorov's prediction formula \eqref{def_nu2};
   see also Remark \ref{rem_ali} for more discussions.
The assumption $\nu^2=0$ in Theorem~\ref{thm_BMD2} is used only to emphasize
that the result addresses the deterministic case; the proof of part (i) only
uses the existence of a non-deterministic arithmetic decimation.
 See also Remark \ref{rem_DEC} for further discussions.

Note that the  condition \eqref{suff_tight} holds whenever $\rho_G\in\ell^1(\Z)$.
That is, we have the following functional Breuer-Major theorem
without the moment assumption \eqref{p>2}.

\begin{corollary} \label{cor_FBM}
Suppose $\rho_G\in\ell^1(\Z)$ and $\varphi\in L^2(\g)$ has 
 Hermite rank $d\geq 1$ as in \eqref{her_varphi}.
Then $V_n \LRA  \s W$ in $\mathbf{D}([0,1])$, 
where $W$ and  $\s^2$ are given  as in 
Theorem \ref{thm_BM}.
\end{corollary}

\begin{remark}\rm
(i) The condition $\rho_G\in \ell^1(\Z)$ is not necessary 
for \eqref{suff_tight}
as illustrated by Example~\ref{exam1}. However, the weaker
condition $\rho_G\in \ell^2(\Z)$ alone does not ensure the
non-determinism of all decimated sequences. Indeed, in
Example~\ref{exam2}, condition~\eqref{DEC} fails: for every $q\geq 2$,
\begin{align}
G_0 \in L^2\big(\Omega,\s\{G_{qj}:j\leq -1\}, \PP\big). \label{DEC_fail}
\end{align}

%
\noi
The construction in Example~\ref{exam2}, and similar examples,
replies on  a rather singular spectral structure
and therefore suggests that the failure of condition~\eqref{DEC}
may be confined to  pathological situations.  
Nevertheless, it demonstrates
that $\ell^2$-summability alone is insufficient to guarantee
non-determinism under decimation.

\smallskip
\noi
(ii)
Note that the condition \eqref{DEC_fail} (i.e., the failure of \eqref{DEC})
means that  the process remains deterministic after every
regular thinning of time. Equivalently, for each $q\ge2$, the decimated
sequence $\{G_{qj}:j\in\mathbb Z\}$ has no innovation: its present value
is already measurable with respect to its own strict past,
and the randomness of the original sequence is not created locally along any
arithmetic time scale, but     encoded globally and redundantly in each
decimated past.

\smallskip
\noi
{\rm(iii)}
The condition \eqref{suff_tight} is {\it only} a sufficient condition for
\eqref{DEC}, as illustrated by Example~\ref{not18}.  The condition \eqref{DEC}
itself appears to be the sharp assumption within the decimation strategy: if no
arithmetic decimation is non-deterministic, then this strategy has no innovation
to exploit.  Thus removing \eqref{DEC}, and hence proving the full BMD principle,
requires ideas beyond decimation.  We believe that \eqref{DEC}, much like the
moment assumption \eqref{p>2}, is an artifact of the proof rather than an
intrinsic obstruction.

\smallskip
\noi
(iv)   Chambers and Slud \cite{CS89} constructed examples of square-integrable
functionals of stationary Gaussian processes for which the CLT
holds but the functional CLT fails.  Their examples belong to
the much larger class of general Wiener-It\^o functionals
$
Y=\sum_{q\ge1}I_q(f_q)$,
where the kernels $f_q$ may depend on the full spectral variables and hence on
infinitely many coordinates of the underlying Gaussian process. 
The main object in their paper is the integral functional
$Z_T(t) = \int_0^{tT} U_s Y ds$
with $\{U_s\}_{s\in\R}$ time-shift operators. 
 In their
counterexample (\cite[pp. 328-329]{CS89}), a dyadic oscillation is inserted into the spectral kernel near
the hyperplane $x_1+\cdots+x_q=0$, so that 
the variance  ratio  $\Var( Z_T(t) ) / \Var(Z_T(1))$
stays bounded above and
below but  fails to converge to $t$.  
Thus, the
 finite-dimensional distributional convergence to a Brownian motion already failed.
This phenomenon is different from the one-coordinate Breuer-Major setting
considered here, where the summands have the form $\varphi(G_k)$ and the chaos
kernels are determined by the Hermite coefficients of $\varphi$ and the
covariance function of $\mathbf{G}$.  Under the Breuer-Major summability condition,
the variance of partial sums over $[0,nt]$ is asymptotically linear in $t$, and
the dyadic spectral oscillation mechanism of \cite{CS89} is not available in our setting. 
With this in mind, we still do not have a counterexample 
to  tightness under the finite-variance assumption. 
On the contrary, our Theorems \ref{thm_BMD1}-\ref{thm_BMD2}
suggest such a counterexample may not exist.

\smallskip
\noi
(v) Our method is tailor-made for the discrete-time setting, in which 
the martingale $M_n$ together with
the filtration $\{\F_k\}_{k\in\Z}$ enters naturally into the picture. 
It is not clear to us how we can adapt our method 
to  the functional Breuer-Major theorem for integral functionals of 
a stationary Gaussian field on $\R^n$
as in \cite{CNN20}. We leave this for  future work.

\smallskip
\noi
(vi) Let us also mention the recent preprint of Altman, Klose, 
and Perkowski \cite{AKP26}, which proves a rough-path version of the
functional Breuer-Major theorem.  They consider vector-valued stationary 
Gaussian sequences and enhance the
partial-sum process by its second-order iterated sums, obtaining convergence
in an $r$-variation rough-path topology.  Their assumptions include
Malliavin-Sobolev regularity, for instance $f_k\in\mathbb{D}^{d,4}(\g)$
in their main theorem, together with the usual Breuer-Major summability of
the cross-covariances.  Their proof is based on Malliavin-calculus and
rough-path tools, including Meyer's inequality, an integration-by-parts
reduction to finite Wiener chaos, and a rough-path tightness criterion.
It is a natural question to weaken the above assumptions by using 
the method in this work.

\end{remark}

The processes defined in \eqref{def_VN} are step functions, 
so the natural path space in the statements above is
$\mathbf{D}([0,1])$.  If one wants to establish   weak convergence on the space 
$\mathbf{C}([0,1])$ of continuous paths on $[0,1]$
equipped with uniform topology,
one can modify the definition of $V_n$ via the polygonal interpolation.

\medskip
\noi
$\bul$ {\bf Linearly interpolated processes and weak convergence in $\mathbf{C}([0,1])$.}   
%
Recall \eqref{def_VN} and set
\begin{align}\label{def_VN2}
\wt V_n(t)
:= 
V_n(t)+ \frac{nt-\lfloor nt\rfloor}{\sqrt n}   \varphi(G_{\lfloor nt\rfloor+1}) 
\end{align}
 for $0\leq t<1$
with
$\wt V_n(1)=V_n(1)$.
Then, $\wt V_n$ 
is a random element of $\mathbf{C}([0,1])$.
In the following, we present a general result
that transfers the weak convergence of $V_n$ on  $\mathbf{D}([0,1])$
to the weak convergence of $\wt V_n$ on  $\mathbf{C}([0,1])$.

\begin{proposition}
\label{prop_C}
Suppose $\varphi\in L^2(\g)$ and the  processes $V_n$ 
and $\wt V_n$ are given as in \eqref{def_VN} and \eqref{def_VN2}.
Let $\{b_n\}_{n\geq 1}$ be a deterministic sequence of real numbers 
with $\liminf_{n\to+\infty} b_n > 0$.
Let $\mathcal{V}$ be a random element of  $\mathbf{C}([0,1])$. 
Then, the following equivalence holds:
\begin{align}\label{D_cvg}
\frac{V_n}{b_n}  \LRA \mathcal{V}
\quad\text{in $\mathbf{D}([0,1])$}
\end{align}
if and only if
\begin{align}\label{C_cvg}
\frac{\wt V_n}{b_n} \LRA  \mathcal{V}
\quad\text{in $\mathbf{C}([0,1])$}.
\end{align}
Consequently, the conclusions of Theorems~\ref{thm_BMD1} and \ref{thm_BMD2}, 
and hence of Corollary \ref{cor_FBM}, 
also hold in $\mathbf{C}([0,1])$ after replacing $V_n$ by $\wt V_n$
and taking $b_n =1$.
\end{proposition}

In the paper \cite{NN20}, the  functional Breuer-Major theorem   
for $\wt{V}_n$
in  $\mathbf{C}([0,1])$ was proved in  Theorem 1.2 therein,
with a proof parallel to that on $\mathbf{D}([0,1])$.
Instead, we provide a much more general equivalence principle.
Such an equivalence principle is standard.  For the reader's convenience, we
include a short proof in Section~\ref{SEC_22}.

\subsection{Application for the  critical fractional-Gaussian-noise}

Our method also applies to the critical normalization in Theorem 5.1 of \cite{NN20},
which assumed \eqref{p>2}. 

 In the notation of the present paper, 
 the limiting variance in this critical regime depends only on 
 the leading Hermite coefficient.

\begin{theorem}
\label{thm_FGN}
Let $\varphi\in L^2(\g)$ have Hermite rank $d\geq1$ 
as in \eqref{her_varphi}.  
Put $H=1-\frac1{2d}$ and let
\[
X_k=B^H_{k+1}-B^H_k,\quad k\in\Z,
\]
be the fractional Gaussian noise associated with a fractional Brownian motion $B^H$. 
That is, 
\[
\E[X_0X_k]=\rho_H(k)
=\frac12\Big(|k+1|^{2H}+|k-1|^{2H}-2|k|^{2H}\Big).
\]
Define
$
Y_n(t):=\frac{1}{\sqrt{n}}\sum_{k=1}^{\lfloor nt\rfloor}\varphi(X_k)
$ for $t\in[0,1]$.
Then,
$
\frac{Y_n}{\sqrt{\log n}}\LRA \s_\ast W$
in  $\mathbf{D}([0,1])$,
where $W$ is a standard Brownian motion and
\begin{align}\label{sstar}
\s_\ast^2
=2d!\,c_d^2\left(\frac{(2d-1)(d-1)}{2d^2}\right)^d.
\end{align}

\noi
The same conclusion holds in $\mathbf{C}([0,1])$ for
the linearly interpolated version of $\frac{Y_n}{\sqrt{\log n}}$
by taking $b_n = \sqrt{\log n}$ in Proposition \ref{prop_C}.
We postpone the proof   to  Section~\ref{SEC_34}.

\end{theorem}

\begin{remark} \rm
As noted in the  proof of Theorem \ref{thm_FGN}, 
the fractional Gaussian noise is non-deterministic for any Hurst index $H\in(0,1)$;
see \eqref{anyH}.
With $\rho_H(k) \sim H(2H-1) |k|^{2H-2}$ for $H\neq 1/2$, we easily see that
$\rho_H\in\ell^d(\Z)$ if and only if $H \in (0,  1  - \frac{1}{2d}) \cup\{\frac12\}$.
 In view of Theorem \ref{thm_BMD1}, the BMD principle holds:
\[
\Big\{ \frac{1}{\sqrt{n}} \sum_{k=1}^{\lfloor nt \rfloor} \varphi( B^H_{k+1}-B^H_k)\Big\}_{t\in[0,1]} 
\LRA \s(H,\varphi) W
\]
for any $H \in (0,  1  - \frac{1}{2d}) \cup\{1/2\}$ and for any $\varphi\in L^2(\g)$
with Hermite rank $d\geq 1$ as in \eqref{her_varphi},
where $\s(H, \varphi)$ is defined as in \eqref{def_s} with $\rho_H$ in place of $\rho_G$.

\end{remark}

\noi
{\bf $\bul$ Organization of the paper.}
The rest of the paper is organized as follows.  In Section~\ref{SEC2}, we collect
the preliminary tools used throughout the proof.  More precisely, in
Section~\ref{SEC_21}, we recall the basic Gaussian analysis needed for the
Ornstein-Uhlenbeck smoothing argument, including Mehler's formula and Nelson's
hypercontractivity.  In Section~\ref{SEC_22}, we record several elementary facts
on the Skorokhod space $\mathbf D([0,1])$, $C$-tightness, 
 and deterministic time changes.  Finally, in
Section~\ref{SEC_23}, we recall a martingale functional central limit theorem
for stationary ergodic martingale difference sequences, in the form needed for
the martingale component of the decomposition \eqref{def_MN}.  The proofs of the
main results are given in Section~\ref{SEC3}, whose detailed organization is
provided at the beginning of that section.

\medskip

\noi
{\bf $\bul$ Acknowledgement.}
We warmly thank Oanh Nguyen for her hospitality during our visit to Brown
University in early June 2026, where most of this work was carried out.
P.M. is funded by Fonds de la Recherche Scientifique (FNRS)
via a FRIA grant.

\section{Preliminaries} \label{SEC2}

Let us fix some notations to be frequently used in this paper. 

\medskip

\noi
$\bul$ {\bf Notations.}
Throughout the paper, $\R_+=[0,\infty)$, $\N=\{1,2,\ldots\}$, $\Z$ denotes the set of integers,
and $\R$ denotes the real line.  We write
$
\T=[-\pi,\pi)
$
for the one-dimensional torus, 
understood modulo $2\pi$.  For $x\in\R$,
$[x]_{\T}$ denotes the unique representative of $x$ modulo $2\pi$ that belongs to
$\T$.  If $A\subset\T$ is an arc or a measurable set, then $|A|$ denotes its
Lebesgue measure, or arc length.  We write $\ind_A$ for the indicator of an event
or set $A$.
For a real-valued random variable $X$ and $p\ge1$, we set
$
\|X\|_p:=\big(\E[|X|^p]\big)^{1/p}.
$
The notation $\LRA$ means convergence in law.  If the underlying path space needs
to be emphasized, we write, for example,
$
X_n\LRA X$
  in $\mathbf D([0,1])$
or
$
X_n\LRA X$
  in $\mathbf C([0,1])$.

For path spaces, $\mathbf D([0,1])$ denotes the space of real-valued c\`adl\`ag
functions on $[0,1]$, endowed with the Skorokhod topology, and
$\mathbf C([0,1])$ denotes the space of continuous functions on $[0,1]$, endowed
with the uniform topology.  For $x\in\mathbf D([0,1])$,
\[
\|x\|_{\sup}:=\sup_{0\le t\le1}|x(t)|.
\]

\subsection{Basic Gaussian analysis} \label{SEC_21}

Recall  that $\mathbf{G} = \{ G_k\}_{k\in\Z}$
 is a centered stationary Gaussian sequence with correlation 
function $\rho_G$ such that $\rho_G(0) =  \E[ G_0^2] = 1$.
Let $\fH$ be a real separable Hilbert space isometric to the $L^2$ Gaussian Hilbert space
generated by the sequence $\mathbf{G}$. 
Then, there is an isonormal Gaussian process $\{X(h)\}_{h\in\fH}$ 
(i.e., centered Gaussian family with covariance    
$\E[ X(h) X(\phi) ] = \langle h, \phi\rangle_\fH$
for any $h, \phi\in\fH$) 
and a sequence of unit vectors $\{ e_k\}_{k\in\Z}$ such that 
$
G_k = X(e_k)$,
$k\in\Z.$
See, e.g., Proposition 7.2.3 in \cite{NP12}. 
Starting from this isonormal framework, one can build the Malliavin calculus
and we refer interested readers to the books \cite{Nua06, NP12}
for more details. For our purpose, we only need 
the following  hypercontractivity inequality due to E. Nelson. 
Let us first define the Ornstein-Uhlenbeck semigroup
$\{P_\tau\}_{\tau\in\R_+}$.
Suppose $F =  f( G_{i_1}, ..., G_{i_m})\in L^2(\Omega)$ for some deterministic measurable 
function $f:\R^m \to \R$
and $i_j\in\Z$,
we define $P_\tau F$ via  Mehler's formula

\noi
\begin{align}\label{Ptau}
P_\tau F = \E\big[  f( G^\tau_{i_1}, ..., G^\tau_{i_m}) | \mathbf{G} \big], 
\qquad
\end{align}

\noi
where $G_j^\tau = e^{-\tau} G_j + \sqrt{1- e^{-2\tau}} G_j' $ with $\{ G'_j\}_{j\in\Z}$ an independent
copy of $\mathbf{G}$. 
Then, by a density argument, 
one can extend the above definition \eqref{Ptau} 
to $L^2(\Omega, \s\{\mathbf{G}\}, \PP)$.
See, e.g., \cite[Theorem 2.8.2]{NP12}.
%
%
 Let us record some basic facts below.

\begin{proposition}[\textsf{Nelson's hypercontractivity}]  \label{prop_1}
Given any $F\in  L^2(\Omega, \s\{\mathbf{G}\}, \PP)$
and $\tau\in(0,\infty)$, we have
$
\| P_\tau F \|_p \leq  \| F \|_2
$
with $p = 1 + e^{2\tau} \in (2, \infty)$. See, e.g., \cite[Theorem 2.8.12]{NP12}.

\end{proposition}

\begin{lemma} \label{lem_1}
Let $\{ G_k \}_{k\in\Z}$ be a centered stationary Gaussian sequence 
with $\rho_G(0)=1$. Let $\F_k = \s\{ G_j: j\leq k \}$,
 $Y_k = \E[ G_k | \F_{k-1}]$, and $\e_k = G_k - Y_k$ for $k\in\Z$.
Then the following results hold. 

\smallskip

\noi
{\rm (i)}  $\{ Y_k \}_{k\in\Z}$ is a  centered stationary Gaussian sequence 
with $\rho_Y(0)  = 1 -\nu^2$ with $\nu = \| G_0 - Y_0\|_2$.
Moreover, $\e_k$, $k\in\Z$, are i.i.d. $\cN(0, \nu^2)$.
 
\smallskip

\noi
{\rm (ii)}  If $\rho_G\in\ell^d(\Z)$ for some $d\geq 2$, then $\rho_Y\in\ell^d(\Z)$ as well. 

\smallskip

\noi
{\rm (iii)} Suppose $\nu^2 < 1$ and define 
$
Z_k:=\frac{Y_k}{\sqrt{1-\nu^2}}.
$
Then, 
for each $k\in\Z$,
\begin{equation}\label{PtauZk}
\qquad\qquad\qquad
\E[\varphi(G_k)\mid\F_{k-1}]
=
P_\tau\varphi(Z_k)
\quad\text{with  
$e^{-\tau}=\sqrt{1-\nu^2}
$.}
\end{equation}

 \end{lemma}

\begin{proof} First it is clear that $\{ Y_k, \e_k\}_{k\in\Z}$ is a centered Gaussian family,
and its stationarity follows immediately from that of $\{ G_k \}_{k\in\Z}$. 
Note that  for $k>j$, $\e_k$ is independent of $\F_{k-1}$ with $\F_j \subset \F_{k-1}$,
then $\e_k$, $k\in\Z$, are i.i.d. Gaussian random variables with mean zero.  
Finally, it is easy to see from stationarity that $\E[ \e_k^2] = \nu^2$.

Next, we prove part (ii).  Using part (i), we can first write
\begin{align} \label{seqa}
\begin{aligned}
\rho_Y(k-m) & = \E[ (G_k - \e_k) ( G_m - \e_m)] 
 = \rho_G(k-m)  -  \E[ G_0 \e_{m-k}] 
\end{aligned}
\end{align}
for $k > m$.  It is clear that   $\E[ G_0 \eps_{i}]$ is square-summable in $i\in\Z$:
the case $\nu = 0$ is trivial while for $\nu^2 > 0$, 
we deduce from Bessel's inequality that
\begin{align}
\sum_{i\in\Z}  |\E[ G_0 \eps_i /\nu ] |^2 \leq  \E[ G_0^2] =1. \label{seqb}
\end{align}
Therefore, due to $\ell^2(\Z)\subset \ell^d(\Z)$ (for $d\geq 2$), 
we deduce immediately that $\rho_Y\in\ell^d(\Z)$.

 Finally, the equation \eqref{PtauZk} follows 
 from part (i) and Mehler formula \eqref{Ptau}
 with 
\[
\E[\varphi(G_k)\mid\F_{k-1}]
=
\int_{\R}
\varphi \big(\sqrt{1-\nu^2}\,Z_k+\nu z\big)\g(dz)
=
P_\tau\varphi(Z_k).
\]

Hence, the proof of Lemma \ref{lem_1} is completed. 
\qedhere
\end{proof}

\subsection{Basic results on tightness} \label{SEC_22}

Let us recall from \cite[Chapter 3]{Bill99} some basics on the space $\mathbf{D}([0,1])$.
Let $\Lambda$ denote the set of 
all strictly increasing continuous bijections
$\lambda:[0,1]\to[0,1]$ satisfying
$
( \lambda(0),  \lambda(1) )=(0,1).$
One standard metric generating the Skorokhod topology   is
\begin{equation}\label{def_dJ1}
\textsf{dist}   (x,y)
:=
\inf_{\lambda\in\Lambda}
\left\{
\|\lambda-\Id\|_{\rm sup}
\vee
\|x-y\circ\lambda\|_{\rm sup}
\right\}.
\end{equation}

\noi
Taking the identity time change $\lambda=\Id$ in \eqref{def_dJ1}, 
we obtain
\begin{align}\label{dist_sup}
\textsf{dist}(x,y)
\leq
\|x-y\|_{\rm sup}.
\end{align}
Let us first recall Billingsley's tightness criterion.
For $x\in\mathbf{D}([0,1])$
and $\dl>0$, define $\|x\|_{\rm sup} = \sup\{ |x(t)|: t\in[0,1] \}$ and
\[
w_x'(\delta)
:=
\inf_{\Pi}
\max_{1\leq i\leq v}
\sup_{s,t\in I_i}
|x(s)-x(t)|,
\]
where the infimum is taken over all  {\it admissible} partitions
$
\Pi = \{ 0=t_0<t_1<\cdots<t_v=1\}
$
such that
$
\inf\{ t_i-t_{i-1}:  1\leq i\leq v\} >\delta
$
and where
$
I_i=[t_{i-1},t_i)
$
for $1\leq i\leq v-1$
while
$
I_v=[t_{v-1},t_v].
$
By Billingsley's criterion (\cite[Theorem 13.2]{Bill99}), 
the tightness of processes  $\{X_n\}\subset\mathbf{D}([0,1])$ 
is equivalent to 
\begin{align}\label{cond_tight}
\begin{cases}
{\rm(i)}\,\, \, {\displaystyle 
\lim_{A\to\infty}
\limsup_{n\to\infty}
 \PP\big(\|X_n\|_{\rm sup}>A\big)=0
 }
  \\[0.5em]
  {\rm and}   
  \\[0.5em]
{\rm (ii)}\,\,\, {\displaystyle  
\lim_{\dl\downarrow0}
\limsup_{n\to\infty}
\PP\big(w_{X_n}'(\delta)>\eps\big)
=0, \,\, \forall \eps > 0.}
\end{cases}
\end{align}

\noi
Note that the addition map is not continuous everywhere in the Skorokhod  topology.
In this paper, we will often need to sum  finitely many tight processes in 
$\mathbf{D}([0,1])$ to reach the desired tightness in $\mathbf{D}([0,1])$.
The lack of continuity of the addition map will not be an issue 
when the processes are $C$-tight, in particular when the limiting process
is continuous. 
Recall also from \cite[p. 124]{Bill99} that 
\begin{align} \label{D2C}
\text{the  Skorokhod convergence to a 
continuous path implies the uniform convergence.
}
\end{align}

\begin{definition}[$C$-tightness] \label{C_tight}
Let $\{X_n\}_{n\ge1}$ be a sequence of random elements in
$\mathbf D([0,1])$, endowed with the Skorokhod  topology.  
We say that
$\{X_n\}_{n\ge1}$ is $C$-tight if
every subsequence of $\{X_n\}$ has a further subsequence converging
in law in $\mathbf D([0,1])$ to a process with continuous paths.
\end{definition}
 
 The above discussion leads to the following lemma,
 which will be used, for example, in the proof of   Proposition \ref{prop_tight_N}.
 
 \begin{lemma} \label{tech_MN}
 Suppose $\{M_n\}_{n\geq 1}, \{N_n\}_{n\geq 1}\subset \mathbf D([0,1])$ are $C$-tight.
 Then, $\{M_n + N_n\}$ is $C$-tight, and in particular tight in $\mathbf D([0,1])$.
 \end{lemma}
 
 \begin{proof} Since $\{M_n\}$ and  $\{N_n\} $ are tight in 
  $\mathbf D([0,1])$, then the sequence of joint laws of  $(M_n, N_n)$ 
  is tight in  $\mathbf D([0,1]) \times  \mathbf D([0,1])$.
 Then, any subsequence of  $\{(M_n, N_n)\}_{n\geq 1}$ 
 admits a further subsequence $\{(M_{n_k}, N_{n_k})\}_{n_k\geq 1}$ 
 that weakly converges to some limit $(M_\infty, N_\infty) \in \mathbf C([0,1])\times \mathbf C([0,1])$.
    By   the Skorokhod representation theorem
(\cite[Theorem~6.7, p.~70]{Bill99}), we can construct processes 
$(M'_{n_k}, N'_{n_k}),  (M'_\infty, N'_\infty)\in \mathbf D([0,1])\times \mathbf D([0,1])$ 
such that $(M_{n_k}, N_{n_k}) = (M'_{n_k}, N'_{n_k})$ in law,
$(M'_\infty, N'_\infty) = (M_\infty, N_\infty)$ in law,
and 
\[
(M'_{n_k}, N'_{n_k})  \xrightarrow[\rm almost \, surely]{n_k\to+\infty} (M'_\infty, N'_\infty) 
\in \mathbf C([0,1])\times \mathbf C([0,1])
\]
with respect to the   Skorokhod  topology.   Then, in view of the fact \eqref{D2C}, 
the above almost sure convergence also takes place with respect to the uniform topology,
and thus  $M'_{n_k} + N'_{n_k}$ has an almost sure limit in  $\mathbf C([0,1])$.
In other words, $\{M_{n_k} + N_{n_k}\}_{n_k}$ weakly  converges  in $\mathbf{D}([0,1])$
to a process with continuous paths. 
This proves the $C$-tightness of $\{ M_n + N_n\}$.
\qedhere

 \end{proof}
 
We are  now ready  to prove 
Proposition \ref{prop_C}.

\begin{proof}[Proof of Proposition \ref{prop_C}]
Since $\liminf_{n\to+\infty} b_n > 0$, we can find $n_0\in\N$ and $\eps_0 > 0$ such that 
$b_n \geq \eps_0$ for any $n\geq n_0$.
By construction of $\wt V_n$, we have for any $n\geq n_0$,
\[
\| \tfrac{1}{b_n} (\wt V_n-V_n)  \|_{\rm sup}
\leq
 \frac{1}{\eps_0 \sqrt n}\max_{1\leq k\leq n}|\varphi(G_k)|,
\]
which converges to zero in  $L^2(\O)$
because of\footnote{To prove \eqref{lim_max}, it is enough to begin with
the truncation: for any $A > 0$,
$\frac{1}{N} \E  \max_{1\leq j \leq N} \varphi(G_j)^2  
\leq \frac{A}{N} + 
\frac{1}{N} \sum_{j=1}^N \E[ 
 \varphi(G_j)^2 \ind_{\{  |\varphi(G_j)|^2 > A \} } ]
= \frac{A}{N} + \E[ 
 \varphi(G_0)^2 \ind_{\{  |\varphi(G_0)|^2 > A \} } ]$
by stationarity.
Then,  sending $N\to+\infty$ first, then $A\to+\infty$,
yields \eqref{lim_max}. }

\noi 
\begin{align}\label{lim_max}
\frac{1}{N} \E\Big( \max_{1\leq j \leq N} \varphi(G_j)^2 \Big) 
\xrightarrow{N\to+\infty} 0.
\end{align}

\noi
That is, we have
\begin{align}\label{diff_P}
\| \tfrac{1}{b_n} (\wt V_n-V_n) \|_{\rm sup}
\xrightarrow{n\to+\infty} 0
\quad\text{in $L^2(\O)$.}
\end{align}

\noi
Consequently, we deduce from \eqref{diff_P} and \eqref{cond_tight}
 that 
  $\{\tfrac{1}{b_n} (\wt V_n-V_n)\}_{n\geq 1}$ is $C$-tight,
and its only possible subsequential limit is the zero process.
  Then, we can deduce from Lemma \ref{tech_MN}
  that
  \begin{center}
   $\{\wt V_n / b_n \}_{n\geq 1}$ is $C$-tight
  if and only if    $\{ V_n / b_n \}_{n\geq 1}$ is $C$-tight.
  \end{center}
  
  Now assume \eqref{D_cvg}. Then, we deduce from 
   \eqref{diff_P}  that 
 the f.d.d. convergence of $\wt V_n / b_n  \LRA \mathcal{V}$,
  and thus
%
%

\noi
\begin{align}\label{D2_cvg}
\wt V_n / b_n  \LRA \mathcal{V}
\quad\text{in }\mathbf{D}([0,1]).
\end{align}

\noi
Therefore,  in view of \eqref{D2C},
we can    upgrade \eqref{D2_cvg} 
to  the convergence  \eqref{C_cvg}
in the continuous path space.\footnote{ To be more precise, 
by the Skorokhod representation theorem
\cite[Theorem~6.7, p.~70]{Bill99}, after passing to a new probability space if
necessary, we may construct random elements $Z_n$ and $Z_\infty$ such that
$Z_n =\wt V_n/b_n$ in law and
$Z_\infty = \mathcal V$ in law, and such that
$Z_n\to Z_\infty$ almost surely with respect to the Skorokhod topology.  Since
$Z_\infty$ has continuous paths, \eqref{D2C} implies that this convergence is in
fact uniform.  As $Z_n$ has continuous paths almost surely, this yields
\eqref{C_cvg}.
}

For the converse direction, assume \eqref{C_cvg}.  Since the embedding
$\mathbf C([0,1])\hookrightarrow \mathbf D([0,1])$ is continuous, we also have
$
\wt{V}_n/b_n\LRA \mathcal{V}$
in $\mathbf{D}([0,1]).$
Moreover, by \eqref{dist_sup} and \eqref{diff_P},
\[
\textsf{dist}(\tfrac{V_n}{b_n},\tfrac{\wt V_n}{b_n} )
\le
 \|
\tfrac{V_n-\wt V_n}{b_n}\|_{\rm sup}
\longrightarrow0
\quad\text{in probability}.
\]
Slutsky's theorem therefore yields \eqref{D_cvg}.
\qedhere

\end{proof}

\begin{lemma}\label{lem_3a}
Fix an integer $q\geq 2$ and $r\in\{0,1,..., q-1\}$. Define
\[
\Theta_n :=
\#\{1\le k\le n:\ k = r \, ({\rm mod}\, q)\}.
\]
For all $n \geq q$ {\rm(}i.e., $\Theta_n>0${\rm)}, define
\[
\theta_{n}(t)
:=
\frac{
\#\{1\le k\le\lfloor nt\rfloor:\ k = r \, ({\rm mod}\, q)\}}{\Theta_n
},
\,\,   t\in[0, 1].
\]
Then
$
\tfrac{\Theta_n}{n}
\xrightarrow{n\to+\infty}
\tfrac1q.
$
Moreover, for all  $n\geq q$, the map
$
\theta_n:[0,1]\to[0,1]
$
is nondecreasing with
$
\big(\theta_n(0),\theta_n(1)\big)=(0,1)
$
such that
$
\sup\big\{  
\left|
\theta_n(t)-t
\right|:  t\in[0,1]  \big\}
\xrightarrow{n\to+\infty}
0.
$
\end{lemma}

\begin{proof}
For any integer  $m\geq 0$, set
$
A_m
:=
\#\{1\le k\le m:\ k\equiv r \, ({\rm mod}\, q)\}.
$
Then, we have 
\[
\Theta_n=A_n
\quad
{\rm and}
\quad
\theta_n(t)
=
\frac{A_{\lfloor nt\rfloor}}{A_n}.
\]
We can write 
$
m=\ell q+a
$
for some 
$\ell\in\N\cup\{0\}$
and $0\le a\le q-1$.
Then, we have 
$
A_m=\ell+\varepsilon_m
$
with 
$ \eps_m\in\{0,1\}$.
Therefore, 
$\big|  A_m-\frac{m}{q} \big| 
= \big| \eps_m - \frac{a}{q} \big| \leq 1.
$
In particular, 
$\big|
\Theta_n-\frac{n}{q} \big|
=
\big|
A_n-\frac{n}{q} \big|
\leq 1$,
yielding 
$
\tfrac{\Theta_n}{n} 
\to
\tfrac1q
$
as $n\to+\infty$.

It remains to prove the uniform convergence $\theta_n(t) \to t$,
 while the rest is straightforward. 
Fix $t\in[0,1]$ and put
$
m=\lfloor nt\rfloor
$
(so that $0\leq nt - m \leq 1$).
Then, we write (for $n\geq q$)
\begin{align*}
\theta_n(t) - t
&= \tfrac{1}{A_n} \big( A_m -  tA_n) 
=  \tfrac{1}{A_n} \big( A_m - \tfrac{m}{q} + \tfrac{m-nt}{q} + t  ( \tfrac{n}{q} -  A_n ) \big), 
\end{align*}

\noi
which is  bounded by $\frac{1}{A_n} ( 2 + \frac{1}{q} )$. 
Therefore,   the desired uniform convergence 
follows   from  $A_n = \Theta_n$
and $
\tfrac{\Theta_n}{n}
\to
\tfrac1q.
$
\end{proof}

  \begin{lemma}\label{lem_3b}
  Let the assumptions of Lemma \ref{lem_3a} prevail. 
  Assume $\{U_n\}_{n\geq 1}$ is  $C$-tight in $\mathbf{D}([0,1])$. 
  Then,  
  $\{U_{\Theta_n} (\theta_n(t)): t\in[0,1]  \}_{n\geq q}$ 
  is a $C$-tight family as well.
 
 \end{lemma}

\begin{proof}

Set $X_n:=U_{\Theta_n}$
and
$Y_n(t):=U_{\Theta_n}(\theta_n(t))=X_n(\theta_n(t))$
for $t\in[0,1]$.
Since $\{U_n\}_{n\geq1}$ is  $C$-tight in $\mathbf D([0,1])$, the subfamily
$
\{X_n\}_{n\geq q}
=
\{U_{\Theta_n}\}_{n\geq q}
$
is also  $C$-tight in $\mathbf D([0,1])$,
meaning any subsequence of 
$\{X_n\}_{n\geq q}$
admits a further subsequence, denoted by $\{Z_n\}$,
such that $Z_n \LRA Z_\infty$ in $\mathbf D([0,1])$
with $Z_\infty\in \mathbf C([0,1])$.
Then, by        Skorokhod representation theorem
(\cite[Theorem~6.7, p.~70]{Bill99}), we can construct processes 
$\{Z'_n\}_{n\leq \infty}$ on some probability space
such that $Z_n = Z_n'$ in law for $n\leq \infty$
and $Z_n' \to Z'_\infty$ with respect to the Skorokhod topology.
Then, in view of \eqref{D2C},
we have $\| Z'_n - Z'_\infty\|_{\rm sup} \to 0$ almost surely
with $Z'_\infty\in \mathbf{C}([0,1])$. 
Due to the uniform convergence of $\theta_n(t) \to t$,
we have $\| Z'_n\circ \theta_n - Z'_\infty\|_{\rm sup} \to 0$.
This proves that any subsequence of $\{Y_n\}$
admits a further subsequence that weakly converges 
in  $\mathbf{D}([0,1])$ to a process with continuous paths. 
Hence, the proof of Lemma \ref{lem_3b} is completed. 
\qedhere
 
\end{proof}

\subsection{Martingale CLTs} \label{SEC_23}

In this section,   we present a well-known result
on functional CLT for stationary ergodic martingale differences.
We say a stationary sequence $\{X_k\}_{k\in\Z}$ is ergodic
if there are no nontrivial shift-invariant events.\footnote{Given a sequence 
$x = \{x_k: k\in\Z\}$, we define the shift operator $T$ by setting $(Tx)_k = x_{k+1}$.
We say a sequence $X$   is stationary if  
$TX$ has the same law as $X$.  We say such a stationary sequence is ergodic
if $\PP(A) \in\{0,1\}$ for any event  $A\in \s\{ X_k : k\in\Z\}$ with 
$T^{-1}A = A$. See, e.g., \cite[Chapter 25]{Kal21}
for more details.}
It is a well-known fact that
a  stationary Gaussian sequence $\mathbf{G}$ is ergodic 
if and only if  its correlation function $\rho_G$ satisfies 
$\frac{1}{n}\sum_{k=1}^n \rho_G(k)^2 \to 0$
as $n\to+\infty$. 
In particular,
$\mathbf{G}$ is ergodic if 
 $\rho_G(k)\to 0$ as $k\to+\infty$ (i.e., when it is mixing).
See, e.g., the classic paper  \cite{Mar49} by G. Maruyama. 
It is also known that the stationarity and ergodicity are 
preserved under measurable transformation and conditioning:

\noi
\begin{align} \label{MEP}
\begin{aligned}
&\text{Let $\{X_k\}_{k\in\Z}$ be a stationary ergodic sequence
with natural filtration $\{ \F^X_k\}_{k\in\Z}$,}\\
&\text{then, $\mathbf{Y}=\{ \E[ \varphi(X_k)| \F^X_{k-1}] \}_{k\in\Z}$
is also stationary and ergodic}
\end{aligned}
\end{align}
for any measurable function $\varphi:\R\to\R$
with $\varphi(X_0)\in L^1(\Omega)$.\footnote{Indeed, 
we first write $\varphi(X_k) = \varphi(X_0)\circ T^k$
and $ \F^X_{k-1} = T^{-k}  \F^X_{-1}$,
with $T$ the shift on the underlying probability space,
then we can write 
$
Y_k = \E[ \varphi(X_0)\circ T^k |  T^{-k} \F^X_{-1} ] = Y_0 \circ T^k,
$
which proves the stationarity of $\{ Y_k\}_{k\in\Z}$.
Note that 
any shift-invariant event of $\{Y_k\}_{k\in\Z}$ is also a shift-invariant
event of  $\{X_k\}_{k\in\Z}$, and therefore is trivial. This proves the 
desired ergodicity. 
}

\begin{proposition}[\textsf{Functional CLT for stationary martingale differences}]
\label{MCLT}
Let 
$(\xi_k,\F_k)_{k\in\Z}$
be a stationary ergodic martingale
difference sequence such that
\[
\s_0:= \sqrt{\E[\xi_0^2] }<\infty.
\]
Define
$
M_n(t)
:=
\frac1{\sqrt n}
\sum_{k=1}^{\lfloor nt\rfloor}\xi_k,$
 $t\in[0,1].$
Then,
$
M_n
\Rightarrow
\s_0 W$
in $\mathbf{D}([0,1])$,
where 
$\mathbf D([0,1])$
is endowed with the Skorokhod  topology
and $W$ is a standard Brownian motion. 
In particular, $\{ M_n\}_{n\geq 1}$ is $C$-tight.

\end{proposition}

For the sake of completeness,
we give a   proof below, 
including the verification of the jump conditions
needed for weak convergence on the  space
$\mathbf D([0,1])$.
We refer interested readers to the survey \cite{Whitt07}
for more on martingale functional CLTs.

\begin{proof}[Proof of Proposition  \ref{MCLT}]
Let us first define a filtration $\mathbb{F}^{(n)} = \{\F_t^{(n)}\}_{t\in\R_+}$
with
$\F_t^{(n)}=
\F_{\lfloor nt\rfloor}.
$
Then, $M_n$
 is a square-integrable c\`adl\`ag martingale
with respect to the filtration  $\mathbb{F}^{(n)}$. 
Next, we set
$
V_k
:=
\E[\xi_k^2| \F_{k-1}]
$
for $k\in\Z.$
Then, in view of the fact \eqref{MEP},
the sequence $\{V_k\}_{k\in\Z}$
 is stationary, ergodic, and
integrable
with
$
\E[V_0]
=
\E[\xi_0^2]
=
\s_0^2.
$
Thus, it follows from
 Birkhoff's ergodic theorem
 that
 \[
\frac1m\sum_{k=1}^mV_k
\xrightarrow[\text{almost surely}]{m\to\infty}
\s_0^2.
\]
See, e.g., \cite[Theorem 25.6]{Kal21}.
Therefore for every  $t\in\R_+$,
\begin{equation*}   
\langle M_n\rangle_t
:=
\frac1n
\sum_{k=1}^{\lfloor nt\rfloor}V_k
\xrightarrow[\text{almost surely}]{n\to\infty}
\s_0^2 \, t.
\end{equation*}

\noi
In fact, the convergence is uniform on compact time intervals. 
Fix $T>0$  and put
\[
R_m
:=
\sum_{k=1}^m(V_k-\s_0^2)
\]
for $m\geq 1$.
Since
$
\frac{R_m}{m} \to 0$
almost surely, 
there exists
$m_0=m_0(\omega,\e)$
 for almost every $\omega$ and every $\e>0$
 such that
$
|R_m|\le \e m$
for all  $m\ge m_0$,
from which we get
\[
\frac1n
\max_{1\le m\le\lfloor nT\rfloor}|R_m|
\le
\frac1n\max_{1\le m<m_0}|R_m|
+
\e T,
\]
with the first term convergent to zero as $n\to+\infty$.
 Letting \(\e\downarrow0\) gives us
\[
\frac1n
\max_{1\le m\le\lfloor nT\rfloor}|R_m|
\xrightarrow[\text{almost surely}]{n\to\infty}
0.
\]

\noi
It follows that
\begin{equation}\label{eq2_MCLT}
\sup_{0\le t\le T}
\left|
\langle M_n\rangle_t-\s_0^2t
\right|
\le
\frac1n
\max_{1\le m\le\lfloor nT\rfloor}|R_m|
+
\frac{\s_0^2}{n}
\xrightarrow[\text{almost surely}]{n\to\infty}
0.
\end{equation}

In view of \cite[Theorem 2.1-(ii)]{Whitt07},
it remains to 
    verify that the jumps of 
$M_n$ and its quadratic variation
$\langle M_n\rangle$ are asymptotically negligible.
%
 Let\footnote{Here, $\Delta x(t) = x(t) - x(t-)$ denotes
 the jump size at time $t$, where $x(t-) = \lim_{s \uparrow t} x(s)$
 is the left limit.}
\[
J_T(M_n)
:=
\sup_{0<t\le T}|\Delta M_n(t)|.
\]
Since
$
\Delta M_n(k/n)=\frac{\xi_k}{\sqrt n},
$
we have
$
J_T(M_n)
=
\frac1{\sqrt n}
\max_{1\le k\le\lfloor nT\rfloor}|\xi_k|.
$
For any fixed $a>0$,
 the elementary inequality
\[
\frac1n
\max_{1\le k\le\lfloor nT\rfloor}\xi_k^2
\le
a^2
+
\frac1n
\sum_{k=1}^{\lfloor nT\rfloor}
\xi_k^2
\ind_{\{|\xi_k|>a\sqrt n\,\}}
\]
gives, by stationarity,
$
\limsup_{n\to+\infty}\E[J_T(M_n)^2]
\leq
a^2.$
Letting \(a\downarrow0\) yields
\begin{equation}\label{eq3_MCLT}
\E[J_T(M_n)^2]
\xrightarrow{n\to+\infty} 
0.
\end{equation}

\noi
Similarly, with
$
\Delta\langle M_n\rangle_{k/n}
=
\frac{V_k}{n},
$
we have
$
J_T(\langle M_n\rangle)
:=
\sup_{0<t\le T}
\Delta\langle M_n\rangle_t
=
\frac1n
\max_{1\le k\le\lfloor nT\rfloor}V_k,
$
and with 
\[
\frac{1}{n} \max_{1\le k\le\lfloor nT\rfloor}V_k
\leq a + \frac{1}{n} \sum_{k=1}^{\lfloor nT\rfloor} V_k \ind_{\{ V_k > an \}}
\]
we have $\limsup_{n\to+\infty} \E\big[ J_T(\langle M_n\rangle) \big] \leq a$
for any $a > 0$.
Therefore, 
\begin{equation}\label{eq4_MCLT}
\E[J_T(\langle M_n\rangle)]
\xrightarrow{n\to+\infty} 0.
\end{equation}
Finally, to  conclude the proof of Proposition \ref{MCLT},
it suffices to apply 
\cite[Theorem 2.1-(ii)]{Whitt07}
with \eqref{eq2_MCLT},
 \eqref{eq3_MCLT},
 and
  \eqref{eq4_MCLT}.
\qedhere
\end{proof}

\section{Proofs of main results} \label{SEC3}

In Section~\ref{SEC_31}, we prove Theorem~\ref{thm_BMD1}, the BMD principle in
the non-deterministic case, by combining the martingale functional CLT with the
Ornstein-Uhlenbeck smoothing of the predictable remainder.  In
Section~\ref{SEC_32}, we first prove Theorem~\ref{thm_BMD2}-(i), the deterministic
case, by decomposing the original partial-sum process into finitely many
decimated subsequences and then applying the non-deterministic result to each
residue class; next we prove
Theorem~\ref{thm_BMD2}-(ii) and give examples illustrating the role and
limitations of the decimation strategy.  Finally, in Section~\ref{SEC_34}, we
prove the critical fractional-Gaussian-noise result,
Theorem~\ref{thm_FGN}.

\subsection{Proof of Theorem \ref{thm_BMD1} (non-deterministic case)} \label{SEC_31}

Let us first recall the decomposition from  \eqref{def_MN}.
Let $\F_k = \s\{ G_j : j\leq k\}$ and 
\[
\Delta_k = \E[ \varphi(G_k) | \F_{k-1} ] 
\quad{\rm and}
\quad
\xi_k = \varphi(G_k) - \Delta_k.
\]
It is immediate to see that 
$(\xi_k,\F_k)_{k\in\Z}$ is a  square-integrable 
martingale difference sequence. 
Moreover, it is a stationary ergodic sequence,
which follows from the fact \eqref{MEP},
and the ergodicity of $\mathbf{G}$ (due to  $\rho_G(k)\to 0$ as $k\to+\infty$).
 Then, we deduce from Proposition \ref{MCLT}
that  $\{M_n\}_{n\geq 1}$ is tight in $\mathbf{D}([0,1])$ and 
\[
M_n \Rightarrow  \mathfrak{a}  W
\quad\text{in  $\mathbf{D}([0,1])$}
\]
with $W$ a standard real Brownian motion
and $\mathfrak{a} = \sqrt{\E[ \xi_0^2  ]}. $\footnote{Using 
the Hermite expansion
\eqref{her_varphi} and \eqref{PtauZk},
it is not difficult to verify that
$\E[\xi_0^2] = \sum_{q\geq d} c_q^2 q! (1 - (1-\nu^2)^q )$.
} 
In particular, $\{ M_n\}_{n\geq 1}$ is $C$-tight. 

\bigskip

Now let us deal with the process $N_n$. 
For its tightness, we need the following   criterion 
from \cite[Lemma 3.1]{NN20}.

\begin{lemma}[Tightness criterion] \label{lem_NN20}
Let $(Y_n)_{n\geq 1}$ be a sequence of random elements 
in $\mathbf{D}([0,1])$. 
Suppose that there is some $p > 2$ and $C \in(0, \infty)$
such that 

\noi
\begin{align}\label{bdd_NN20}
\| Y_n(t) - Y_n(s) \|_{p} \leq C \Big(  \frac{  \lfloor nt \rfloor - \lfloor ns \rfloor}{n} \Big)^{\frac12}
\end{align}
for any $0 \leq s \leq t \leq 1$
and for any $n\geq 1$.
Then,  $\{Y_n\}_{n\geq 1}$ is tight
in $\mathbf{D}([0,1])$. 
\end{lemma}
 
Note that in the paper \cite{NN20},
Nourdin and Nualart applied  
the above Lemma \ref{lem_NN20},
together with an elegant application 
of Meyer's inequality to establish 
the bound \eqref{bdd_NN20} for
  $\| V_n(t) - V_n(s) \|_{p}$,
 by assuming $\varphi\in L^p(\g)$
 for some $p>2$.
 Note that 
 
 \noi
\[
N_n(t) = \frac{1}{\sqrt{n}} \sum_{k=1}^{\lfloor nt \rfloor} P_\tau \varphi(Z_k)
\quad\text{with $e^{-2\tau} = 1- \nu^2$}, 
\]
where $\rho_Z\in\ell^d(\Z)$ (provided $\rho_G\in\ell^d(\Z)$ for some $d\geq 2$)
and $P_\tau \varphi  \in L^p(\g)$ for $p = 1 + e^{2\tau}>2$
by Lemma \ref{lem_1} and Proposition \ref{prop_1}.
As a result, we have the tightness for $\{N_n\}_{n\geq 1}$ by 
invoking the result of Nourdin and Nualart
with $d\geq 2$.
In the following, we will present a simple proof 
{\it without} using Meyer's inequality.

\begin{proposition} \label{prop_tight_N}
Suppose $\rho_G\in\ell^d(\Z)$ for some $d\geq 2$ 
and $\nu^2 \in ( 0, 1]$.
Then, 
 $\{N_n\}_{n\ge1}$ is
tight in 
$\mathbf{D}([0,1])$
and
\[
N_n\LRA \mathfrak{b} W
\quad\text{in  $\mathbf{D}([0,1])$,} 
\]

\noi
where 
$
\mathfrak{b}^2
=
\sum_{q=d}^{\infty}q!c_q^2 
\sum_{k\in\Z}  \big[(1- \nu^2) \rho_Z(k) \big]^q \in [0,\infty).$

\end{proposition}

\begin{proof}

If $\nu^2=1$, then  $\DL_k=0$ and thus the result is trivial.
Throughout this proof, we  assume 
$0<\nu^2<1$.

By Lemma \ref{lem_1},
 $\{Z_k\}_{k\in\Z}$ is a centered stationary Gaussian sequence with covariance
$\rho_Z\in\ell^d(\Z).$ Then f.d.d. convergence of 
$N_n \Rightarrow \mathfrak{b}W$ follows from Theorem \ref{thm_BM}. 
It remains to prove the tightness of $\{N_n\}_{n\geq 1}$.
For $t > s$, we first write 

\noi
\begin{align} \label{mark1}
N_n(t) - N_n(s) = \frac{1}{\sqrt{n}} \sum_{\lfloor ns \rfloor < k \leq \lfloor nt \rfloor  } P_\tau \varphi(Z_k) 
 = P_\tau\bigg(  \frac{1}{\sqrt{n}} \sum_{\lfloor ns \rfloor < k \leq \lfloor nt \rfloor  }   \varphi(Z_k)  \bigg)
\end{align}
and then we apply  hypercontractivity (Proposition  \ref{prop_1})
to obtain, with $p =1 +e^{2\tau} > 2$, 

\noi
\begin{align*}
\big\| N_n(t) - N_n(s) \big\|_{p} 
\leq   \bigg\|  \frac{1}{\sqrt{n}} \sum_{\lfloor ns \rfloor < k \leq \lfloor nt \rfloor  }   
         \varphi(Z_k)  \bigg\|_{2}.
\end{align*}
Without loss of generality, 
we can assume $\lfloor ns\rfloor < \lfloor nt\rfloor $ 
and write with $I_{s,t} = (\lfloor ns \rfloor , \lfloor nt \rfloor    ]$,
\begin{align}
& \quad \bigg\|  \frac{1}{\sqrt{n}} \sum_{\lfloor ns \rfloor < k \leq \lfloor nt \rfloor  }   \varphi(Z_k)
              \bigg\|_{2}^2
= \frac{1}{n} \sum_{k, j\in I_{s,t}} \E\big[ \varphi(Z_k) \varphi(Z_j) \big] 
\notag\\
& =\frac{\lfloor nt\rfloor - \lfloor ns\rfloor }{n}  \sum_{q\geq d} c_q^2 q!  
 \bigg( \frac{1}{\lfloor nt\rfloor - \lfloor ns\rfloor} \sum_{k, j\in I_{s,t}}  
 \rho_Z(k-j)^q \bigg) 
 \notag\\
 &\leq \frac{\lfloor nt\rfloor - \lfloor ns\rfloor }{n}  \sum_{q\geq d} c_q^2 q!  
 \sum_{|m|\leq \lfloor nt\rfloor - \lfloor ns\rfloor}  
 \frac{\lfloor nt\rfloor - \lfloor ns\rfloor - |m|}{\lfloor nt\rfloor - \lfloor ns\rfloor} |\rho^q_Z(m)| 
 \label{mark2}  \\
 &\leq  \frac{\lfloor nt\rfloor - \lfloor ns\rfloor }{n}  \sum_{q\geq d} c_q^2 q!  
 \| \rho_Z\|^q_{\ell^q(\Z)} \leq \frac{\lfloor nt\rfloor - \lfloor ns\rfloor }{n}  \| \varphi\|^2_{L^2(\g)} 
 \| \rho_Z\|^d_{\ell^d(\Z)}.
 \notag
\end{align}

\noi
This concludes the proof of tightness of $\{ N_n\}_{n\geq 1}$
and  $N_n \LRA \mathfrak{b}W$ in $\mathbf{D}([0,1])$.
\qedhere
 \end{proof}

So far, we have proved $C$-tightness of $\{M_n\}$ and $\{N_n\}$
when $\rho_G\in\ell^d(\Z)$ for $d\geq2$ and $\nu^2>0$.
Then, we can deduce from Lemma \ref{tech_MN}
that $\{V_n\}_{n\geq 1}$ is $C$-tight in this case.  
 The finite-dimensional convergence of $V_n$ is exactly Theorem~\ref{thm_BM}; 
hence tightness yields the desired convergence in $\mathbf D([0,1])$.

It remains to deal with the case $d=1$. 
If $\varphi(x) = c_1x$, the result follows easily from the 
standard Gaussian computations with $\rho_G\in\ell^1(\Z)$
and Lemma \ref{lem_NN20}.
Otherwise, we can write 
$\varphi(x) = c_1x + \widehat{\varphi}(x)$ with $\widehat\varphi$ having 
Hermite rank $\widehat d \geq 2$.
In this way, we write 
\[
V_n(t) = \frac{1}{\sqrt{n}} \sum_{k=1}^{\lfloor nt \rfloor} c_1 G_k + \widehat V_n(t),
\]
where $\widehat V_n(t)  =  \frac{1}{\sqrt{n}} \sum_{k=1}^{\lfloor nt \rfloor} \widehat\varphi(G_k)$.
By the previous paragraph, we have the $C$-tightness of $\{\widehat V_n\}_{n\geq 1}$ and 
the Gaussian process component.
Thus, $\{V_n\}_{n\geq 1}$ is $C$-tight by Lemma \ref{tech_MN}.
Combining tightness with Theorem~\ref{thm_BM} proves the 
desired functional convergence in  $\mathbf{D}([0,1])$.

Hence, the proof of Theorem \ref{thm_BMD1} is completed. 
\hfill $\square$

\subsection{Proof of Theorem \ref{thm_BMD2} (deterministic case)} \label{SEC_32}

Let us first prove part (i) of Theorem \ref{thm_BMD2}. 
Recall that
for an integer $q\geq 2$
and
 $r\in\{0,\ldots,q-1\}$, we define the decimated Gaussian sequence
\begin{equation}\label{def_GR}
\text{$G_j^{(r)}:=G_{qj+r},$
 $j\in\Z$}.
\end{equation}
Its covariance is
\begin{equation}\label{def_CQ}
\mathcal{C}_q(h)
:=
\E[G_0^{(r)}G_h^{(r)}]
=
\rho_G(qh),
\end{equation}
which does not depend on $r$ due to stationarity of $\mathbf{G}$.

\begin{proof}[Proof of Theorem \ref{thm_BMD2}-\rm(i)]

As noted above,  the law of the decimated sequence
$\{G_{qj+r}:j\in\Z\}$ does not depend on $r$.  
Thus,  \eqref{DEC} implies that
each residue-class decimation is non-deterministic.
That is, 
one can find some integer $q\geq 2$ such that
 for each $r\in\{0, ..., q-1\}$,
the decimated sequence
$
\{G^{(r)}_j, j\in\Z\}
$
is non-deterministic with   covariance  
$
\mathcal{C}_q\in\ell^d(\Z)
$.
Define 

\noi
\begin{align*}  
U_N^{(r)}(t)
:=
\frac1{\sqrt N}
\sum_{j=1}^{\lfloor Nt\rfloor}
\varphi(G_{qj+r}),
\quad  t\in[0,1].
\end{align*}

\noi
Then, by Theorem \ref{thm_BMD1},
$\big\{ U^{(r)}_N \big\}_{N\geq 1}$
is  $C$-tight in 
$\mathbf{D}([0,1])$.

For $r\in\{ 0,\ldots,q-1\}$, define
$
V_{n,r}(t)
:=
\frac1{\sqrt n}
\sum_{k=1}^{\lfloor nt\rfloor}  
\varphi(G_k) \ind_{\{ k =  r  \, ({\rm mod}\, q)\}}.
$
Then,

\noi
\begin{align}\label{decompV}
V_n(t)=\sum_{r=0}^{q-1}V_{n,r}(t).
\end{align}

\noi
For each fixed $r$, $V_{n,r}$ differs, up to at most two boundary terms 
 divided by $\sqrt n$, from a deterministic time change and scaling
of \(U_N^{(r)}\), with 
$N/n\to 1/q$. 
More precisely, we define 
\[
m_{n,r}:=
\#\{1\le k\le n:  k =  r  \, ({\rm mod}\, q)\}
\quad
{\rm and}
\quad
\theta_{n,r}(t)
:=
\frac{ m_{\lfloor nt\rfloor, r} }{m_{n,r}}.
\]
In the rest of the proof, we assume $n\geq q$
so that  $m_{n,r} > 0$ and then  the function $\theta_{n,r}$ is well defined. 
From Lemma \ref{lem_3a}, we have $\frac{m_{n,r}}{n} \to \frac1q$ as $n\to+\infty$,
and
$\theta_{n,r}: [0,1] \to [0,1]$ is nondecreasing 
with $\theta_{n,r}(0)=0$, $\theta_{n,r}(1)=1$,
and
$\sup\{ | \theta_{n, r}(t)  - t  |  : t\in[0,1] \}
\xrightarrow{n\to+\infty} 0.
$
Then, it follows from Lemma \ref{lem_3b} that 
$\big\{ U_{m_{n, r}}^{(r)}(\theta_{n, r}(t)): t\in[0,1] \big\}_{n\geq q}$
is  $C$-tight. 
Next, we write 
\[
V_{n,r}(t)
=
\sqrt{\frac{m_{n,r}}{n}}\,
U_{m_{n,r}}^{(r)}(\theta_{n,r}(t))
+
R_{n,r}(t),
\]
where $R_{n,r}$ is the boundary correction term given by 
\[
R_{n, r}(t) = \frac{1}{\sqrt{n}}\big[ \varphi(G_r) - \varphi(  G_{r + q m_{\lfloor nt \rfloor, r} } ) \big]
\ind_{\{ 1\leq r \leq q-1\}}.
\]
From $qm_{n,r}/n\to 1$ and \eqref{lim_max}, 
we can easily deduce 
\begin{align}\label{L2_R}
\lim_{n\to+\infty}\E\bigg[ \sup_{0\le t\le1}|R_{n,r}(t)|^2 \bigg]
= 0.
\end{align}
Then,  it follows   from \eqref{L2_R}
that $\{R_{n,r}\}_{n\geq q}$ is $C$-tight. In fact, the tightness criterion \eqref{cond_tight}
can be applied here with the easy fact $w'_x(\dl) \leq 2 \|x\|_{\rm \sup}$.

\medskip

Hence, $\{V_{n,r}\}_{n\geq q}$ is $C$-tight for each $r$,
and so is $\{V_n\}_{n\geq q}$ in view of \eqref{decompV}.  
Hence, the proof of Theorem \ref{thm_BMD2}-(i) is completed
by combining the f.d.d. convergence in Theorem \ref{thm_BM} 
and the obtained tightness.
\qedhere

\end{proof}

Next, we prove part (ii) of Theorem \ref{thm_BMD2},
which essentially asserts that a deterministic
stationary Gaussian sequence with absolutely summable covariance 
becomes non-deterministic after sufficiently large decimation.

\bigskip
\noi 
{\it Proof of Theorem \ref{thm_BMD2}-\rm(ii)}.
By assumption, the $\fq$-decimated covariance $\mathcal{C}_{\fq}(h)=\rho_G(\fq h)$ belongs to $\ell^1(\Z)$.  Thus, for each residue class modulo $\fq$, the corresponding decimated Gaussian sequence has an absolutely summable covariance.  It is enough to prove Theorem \ref{thm_BMD2}-(ii)
 in the special case $\rho_G\in\ell^1(\Z)$: if an integer $q_0$ works for every residue class of any stationary Gaussian sequence with absolutely summable covariance, then applying that result to each of the $\fq$ residue classes gives a decimation of the original sequence with step $q=\fq q_0$.  This $q$ is divisible by $\fq$ and works for all residue classes modulo $q$.

{\bf Throughout this proof, we will assume}
 $\rho_G\in\ell^1(\Z)$.
 Then, 
the spectral density
\begin{equation}\label{eq_m}
\fm(\lambda)
:=
\sum_{k\in\Z}\rho_G(k)e^{-ik\lambda}
\end{equation}
is {\it continuous} and nonnegative on
the unit circle  $\T=[-\pi,\pi)$
such that 
$
\frac1{2\pi}\int_{-\pi}^{\pi}\fm(\lambda)\,d\lambda
=
\rho_G(0)=1.
$
Therefore, there exist a nonempty  open arc
$I\subset\T$ and a constant $a>0$ such that
\begin{equation}\label{low_ma}
\text{$\fm(\lambda)\geq a$ 
for  any $\lambda\in I$}.
\end{equation}

Choose an integer $q\geq 2$ such that
\begin{equation}\label{q_chosen}
\frac{2\pi}{q}< \vert I \vert,
\end{equation}
where \(|I|\) denotes the length of the arc \(I\).
Recall
$
\mathcal{C}_q\in\ell^1(\Z)
$
and the associated spectral density $\fm_q$ is given by 
\begin{align}\label{mq1}
\fm_q(\lambda) 
= \sum_{k\in\Z}\mathcal{C}_q(k)e^{-ik\lambda}
=\sum_{k\in\Z}\rho_G(qk)e^{-ik\lambda}
\end{align}
for $\lambda\in\T$. 
Next, we present an equivalent expression for $\fm_q$
in terms of $\fm$,
and we write 
$[x]_{\T}$ for the unique representative of 
$x$ modulo $2\pi$ that
belongs to \(\T=[-\pi,\pi)\).

\begin{lemma}   \label{lem_2a}
Let $\mathbf G$ have an absolutely continuous spectral measure with density
$\mathfrak m\in L^1(\mathbb T)$. For $q\ge2$, the $q$-decimated sequence
has spectral density
\begin{align}\label{eq_mq}
\mathfrak m_q(\theta)
=
\frac1q\sum_{\ell=0}^{q-1}
\mathfrak m\left(\left[\frac{\theta+2\pi\ell}{q}\right]_{\mathbb T}\right)
\end{align}
for a.e. $\theta\in\mathbb T$.
If, in addition, $\rho_G\in\ell^1(\mathbb Z)$, then the equality holds
pointwise for the continuous versions.
\end{lemma}

\begin{proof}
The formula \eqref{eq_mq} clearly defines a function $\fm_q\in L^1(\T)$.
For $h\in\Z$,

\noi
\begin{align}    \label{AL1}
 \frac{1}{2\pi} \int_{\T}e^{ih\theta}  \fm_q(\theta)\,d\theta
&= \frac{1}{2\pi q}  \sum_{\ell=0}^{q-1}   \int_{\T} e^{ih\theta}
         \fm  \big(   \big[\tfrac{\theta+2\pi\ell}{q}\big]_{\T}  \big) d\theta .
\end{align}
For the $\ell$-th integral, set
$   \lambda  =   \big[\tfrac{\theta+2\pi\ell}{q}\big]_{\T}$.
As $\theta$ runs over $\T$, the variable $\lambda$ runs over an arc
$I_\ell$ of length $2\pi/q$, and the arcs
$I_0,\ldots,I_{q-1}$ form a partition of $\T$ modulo endpoints.
Moreover,
\[
        \theta=q\lambda-2\pi\ell
        \quad \mathrm{mod}\, 2\pi,
        \qquad d\theta=q\,d\lambda .
\]
Since $h\in\Z$, we have 
$ 
e^{ih\theta}  =    e^{ih(q\lambda-2\pi\ell)}   =   e^{iqh\lambda}.
$
Therefore, from \eqref{AL1},
\begin{align}
        \frac{1}{2\pi}
        \int_{\mathbb T}e^{ih\theta}\mathfrak m_q(\theta)\,d\theta
        =
        \frac{1}{2\pi}
        \sum_{\ell=0}^{q-1}
        \int_{I_\ell}
        e^{iqh\lambda}\mathfrak m(\lambda)\,d\lambda
        =
        \frac{1}{2\pi}
        \int_{\mathbb T}
        e^{iqh\lambda}\mathfrak m(\lambda)\,d\lambda
        =
        \rho_G(qh).
        \label{AL2}
\end{align}

\noi
Hence, $\fm_q$ is a spectral density of the $q$-decimated sequence. 
Finally, if $\rho_G\in\ell^1(\Z)$, then
$\fm$
is continuous, and $\fm_q$ in  \eqref{eq_mq}  is also
continuous. Since continuous functions on $\T$ are determined by
their Fourier coefficients, the identity   holds   for every
$\theta\in\T$.
\qedhere
\end{proof}

\begin{remark}[\textsf{Spectral aliasing}]
\label{rem_ali}
\rm
The identity  \eqref{eq_mq}  in Lemma \ref{lem_2a} is the standard spectral aliasing formula
for downsampling. It is familiar in the spectral analysis of wavelet and
filter-bank transforms: after a filtering step, keeping one point out of every
$q$ folds the spectrum by averaging $q$ shifted and rescaled copies of the
original density. The dyadic case $q=2$ appears explicitly in the recursive
spectral-density relations for discrete wavelet coefficients in
Mielniczuk and Wojdyllo \cite{MW05}. Related central limit theorems
for arrays of decimated linear processes, motivated in part by spectral-density
and wavelet long-memory estimation, were studied by Roueff and Taqqu
\cite{RT09a,RT09b}.
%
The role of aliasing here is different from its role in wavelet estimation.
We do not use it to analyze a scalogram or a spectral estimator. Instead, we
use aliasing to force positivity of the spectral density of a decimated Gaussian
sequence. Indeed, once a positive arc of the original spectral density is hit
by every $q$-point preimage grid, the aliased density has a positive lower
bound. This is precisely the step leading to \eqref{low_mqa} below, 
and hence to non-determinism through Kolmogorov's prediction formula \eqref{def_nu2}.

\end{remark}

\begin{lemma}\label{lem_mqb}
Assume $\rho_G\in\ell^1(\Z)$ and 
let  $q\geq 2$ be   as in \eqref{q_chosen}.
For any  $r\in\{ 0,\ldots,q-1\}$,
the decimated
sequence $\{G_j^{(r)}\}_{j\in\Z}$ 
is
non-deterministic in the one-step prediction sense.

\end{lemma}

\begin{proof}
Fix $\theta\in\T$. The $q$ points
$
\big[
\tfrac{\theta+2\pi\ell}{q}
\big]_{\T}$,  $\ell=0,\ldots,q-1$,
are exactly the \(q\) preimages of \(\theta\) under the map
$
\lambda\in\T\mapsto q\lambda\pmod{2\pi}.
$
They form a grid on the circle with mesh \(2\pi/q\). 
By  \eqref{q_chosen},
this grid must intersect the arc \(I\). Hence, for every \(\theta\in\T\),
there exists \(\ell=\ell(\theta)\) such that
$
\big[
\tfrac{\theta+2\pi\ell}{q}
\big]_{\T}
\in I$
so that 
$
\fm\big(
\big[
\tfrac{\theta+2\pi\ell}{q}
\big]_{\T}
\big)
\geq a
$
in view of   \eqref{low_ma}.
Therefore, we deduce from Lemma \ref{lem_2a}
that 
\begin{equation}\label{low_mqa}
\inf_{\theta\in\T}\fm_q(\theta) \geq \frac{a}{q} > 0,
\end{equation}

\noi
and thus
\begin{align}\label{def_nuq}
\nu_q^2
:=
\exp\left\{
\frac1{2\pi}
\int_{-\pi}^{\pi}\log\fm_q(\theta)\,d\theta
\right\}
>0.
\end{align}
Hence, we can conclude the proof by applying
Kolmogorov's prediction formula \eqref{def_nu2}.
\qedhere

\end{proof}

Now we have fully proved Theorem \ref{thm_BMD2}.

\begin{remark}[\textsf{On the decimation technique}]
\label{rem_DEC}
\rm
Decimation, or decomposition into arithmetic subsequences, has appeared
in several Gaussian limit-theorem arguments. For instance, this technique
is used in Bardet and Surgailis \cite{BS13b} 
(the paper referred to as Bardet-Surgailis (2013b) in \cite{LPV23})
and later in Li-Pakkanen-Veraart \cite{LPV23}. 
In these works, the basic role of decimation is to split a partial sum into
finitely many residue classes so that, inside each class, indices are separated
by a prescribed distance. This separation makes the relevant Gaussian
correlations small enough for the available moment bounds to apply.
%
Our use of decimation is different. Our goal is not merely to separate
highly correlated summands. In the deterministic case, the
martingale-predictable decomposition is unavailable along the full sequence,
because the one-step innovation variance vanishes. We instead choose a
decimation for which the aliased spectral density becomes bounded away from
zero, and then Kolmogorov's prediction formula yields non-determinism for
each residue class. Thus, in the present paper, decimation is used as a
prediction-theoretic device: it recovers innovations on arithmetic subsequences,
after which the non-deterministic BMD principle can be applied and the
finitely many residue classes recombined.

\end{remark}

In the following, we present a few examples on decimations.

\begin{example}[\textsf{$\rho_G\in \ell^1(\Z)$ is not necessary 
for non-determinism via decimation}]
\rm
\label{exam1}
Consider the spectral density 
\[
\fm(\lambda)
=
2\,\ind_{[-\pi/2,\pi/2) }(\lambda),
\quad
\lambda\in[-\pi,\pi).
\]
Then,  the associated  covariance function is
\[
\rho_G(k)
=
\frac1{2\pi}
\int_{-\pi}^{\pi}
e^{ik\lambda}\fm(\lambda)\,d\lambda
=
\frac{2}{\pi}\frac{\sin(k\pi/2)}{k}
\]
for $k\neq 0$ and $\rho_G(0)=1$.
Thus, 
$
\rho_G\in\ell^2(\Z) \setminus \ell^1(\Z).
$
Note that the original sequence $\mathbf{G}$ is deterministic, 
because $\fm$ vanishes on a set
of positive Lebesgue measure. 
However, for $q=2$,
\[
\fm_2(\theta)
=
\frac12
\left[
\fm\left(\frac{\theta}{2}\right)
+
\fm\left(
\left[\frac{\theta+2\pi}{2}\right]_{\T}
\right)
\right]
=
1
\]
for   every \(\theta\in[-\pi,\pi)\). 
Therefore, the two decimated
sequences
$
(G_{2j})_{j\in\Z}
$
and
$
(G_{2j+1})_{j\in\Z}
$
are non-deterministic; in fact, they are i.i.d. standard Gaussian
sequences.
\end{example}

\begin{example}[\textsf{$\rho_G\in \ell^2(\Z)$ alone 
does not guarantee 
non-deterministic decimation}]
\label{exam2}
\rm
In the following, 
we construct an absolutely continuous spectral density $\fm$ with
$\rho_G\in\ell^2(\Z)$
such that no decimation becomes non-deterministic.

Let $\T=[-\pi,\pi)$, understood modulo $2\pi$. 
For $q\geq 2$,
choose the open arc
$
B_q
:= ( -2^{-q-1}, 2^{-q-1} )
$
and 
define
$
E_q
:=
\left\{
\lambda\in\T:
q\lambda \,\, ({\rm mod}\, 2\pi)  \in B_q
\right\}.
$
It is not difficult to see that\footnote{The set $E_q$ is the union 
of $q$ equally spaced open arcs,
   each having  length $\frac{1}{q2^q}$,
and every gap between two consecutive arcs has length
strictly smaller than $2\pi/q$. 
Now let $I\subset\T$ be any nonempty open arc. 
Choose $q\geq 2$ so large
that
$
\frac{2\pi}{q}<\vert I \vert.$
Since every connected component of 
$\T\setminus E_q$ has length strictly
smaller than $|I|$,  
the arc $I$ cannot be contained in $\T\setminus E_q$.
Therefore $I\cap U\neq\emptyset$.
Hence $U$ is dense in $\T$, that is, \eqref{claim_U} is proved.}

\noi
\begin{align} \label{claim_U}
U
:=
\bigcup_{q=2}^{\infty}E_q
\text{\quad is dense in $\T$.}
\end{align}

\noi
Moreover, since $E_q$ is the union of $q$ arcs, each of length $\frac{1}{q2^q}$, we have
$
|U|
\leq
\sum_{q=2}^{\infty}|E_q| 
<2\pi.
$
Then, the set 
$
A:=\T\setminus U
$
is closed, has positive Lebesgue measure, and has empty interior.

Define
$
\fm(\lambda)
=
\frac{2\pi}{|A|}\ind_A(\lambda)$
for $\lambda\in\T$.
Then, $\fm$ is well defined to be a spectral density on $\T$.
Due to  $\fm\in L^2(\T)$, Parseval's identity yields
$
\rho_G\in\ell^2(\Z).
$
On the other hand, if $\rho_G$ were in $\ell^1(\Z)$,
the spectral density $\fm$ would admit  a continuous 
version, which is impossible. 
That is, we just proved 
\[
\rho_G\in\ell^2(\Z)\setminus \ell^1(\Z).
\]
 Note that 
the original sequence $\mathbf{G}$ is deterministic 
in the one-step prediction sense, 
since $\fm=0$ on $U$ with $|U | > 0$.

Next, we   show that every decimation remains deterministic. 
Fix any $q\geq 2$. If
$
\theta\in B_q,
$
then 
$[\frac{\theta+2\pi\ell}{q}]_{\T} \notin A$
 for any  $\ell\in\{0,\ldots,q-1\}$,
 because $\lambda\in U$ whenever 
$
q\lambda\equiv\theta\pmod{2\pi}\in B_q.
$
Consequently,
$\fm\big([ \frac{\theta+2\pi\ell}{q}]_{\T} \big)=0$
for every $\ell\in\{0,\ldots,q-1\}$.
By Lemma \ref{lem_2a},
$
\fm_q(\theta)=0$
for any  $\theta\in B_q$.
Since $B_q$ has positive Lebesgue measure,
we have $\nu_q = 0$ with $\nu_q$ defined as in \eqref{def_nuq},
 and thus
the $q$-decimated sequence is deterministic. 
This example shows that the condition
$
\rho_G\in\ell^2(\Z)
$
alone does not imply the spectral decimation condition.
\end{example}

 \begin{example}[\textsf{Non-deterministic decimations without} \eqref{suff_tight}] \label{not18}
 \rm

Define for $\lambda\in\T$,
\[
\fm(\lambda)
=
\frac{\pi}{\alpha}\mathbf 1_{[-\alpha,\alpha]}(\lambda)
\text{\quad with $\alpha=\pi/\sqrt2$}. 
\]
The corresponding covariance function is
\[
\rho_G(k)
=
\frac{\sin(\alpha k)}{\alpha k} \ind_{\{ k\neq 0 \}} + \ind_{\{ k =0  \}}.
\]

\noi
 Since $\fm$ vanishes on a set of positive Lebesgue measure,
the corresponding sequence $\mathbf{G}$ is deterministic in the one-step prediction sense
and
$\rho_G\in\ell^2(\Z)$.

In the following, we will show that the condition  
\eqref{suff_tight} fails. 
 Indeed, for every integer $q \geq 2$,
 $\alpha q/(2\pi)=q/(2\sqrt2)$ is irrational, 
 then  
 {\it Weyl's equidistribution theorem} implies that
$\{\alpha qk\ ({\rm mod}\ 2\pi)\}_{k\ge1}$ is equidistributed on $[0,2\pi)$;
see \cite[Theorem~2.1, p.~107]{SS03}.
Let  $A:=\{\theta\in[0,2\pi):|\sin\theta|\ge 1/2\}.$
Since $|A|=4\pi/3$, equidistribution gives us

\noi
\begin{align}\label{EquiD}
\frac{1}{N} \# (E_q \cap \{1, 2, ..., N\})
\to
2/3 \quad\text{as $N\to+\infty$,}
\end{align}

\noi
where  $E_q := \{k\in\N:\alpha qk\ ({\rm mod}\ 2\pi)\in A\}$.
Thus,  the set $E_q\cap \{1, 2, ..., N\}$ has positive natural density, and hence the
corresponding harmonic sub-series diverges:\footnote{More precisely,
by \eqref{EquiD},
 $\exists N_0\in\N$ such that  $\forall N\ge N_0$,
$
\#(E_q\cap\{1,\ldots,N\})\ge \frac13 N.
$
Write $E_q=\{n_1<n_2<\cdots\}$.  Then, $\exists j_0 \geq 1$ sufficiently large s.t.  $\forall j\geq j_0$,
$
 \frac{1}{n_j} \#(E_q\cap\{1, 2, \ldots,n_j\})
\geq \frac{1}{3}$,
and hence $n_j\le 3j$.  Therefore
$
\sum_{k\ge1}\frac1k\ind_{\{ \al qk\, ({\rm mod}\, 2\pi)\in A\}}
=
\sum_{j\ge1}\frac1{n_j}
\ge
\frac13\sum_{j\ge j_0}\frac1j
= + \infty.$          }

\noi
\begin{align*}
\sum_{k=1}^{\infty}\frac{|\sin(\alpha qk)|}{k}
\geq
\frac{1}{2}
\sum_{k\geq 1}
\frac{1}{k} \ind_{\{\al qk\ ({\rm mod}\ 2\pi)\in A\}}
=
\infty.
\end{align*}
Consequently, we have 
$
\sum_{k=1}^{\infty}  | \rho_G(q k )|
=\frac{1}{q \al} \sum_{k=1}^{\infty}    \frac{|\sin(\alpha q k)|}{k}
=
\infty.
$
That is,  \eqref{suff_tight} fails.
Nevertheless, the $2$-decimated sequences
are non-deterministic. 
Indeed, the spectral density of the
$2$-decimated sequence is
$
\fm_2(\theta)
=
\tfrac{1}{2}
\left[
\fm\left(\tfrac{\theta}{2}\right)
+
\fm\left(\left[\tfrac{\theta+2\pi}{2}\right]_{\T}\right)
\right]
$
for almost every $\theta\in\T$,
in view of Lemma \ref{lem_2a}.
The   points $\frac{\theta}{2}$ and $[\frac{\theta+2\pi}{2}]_{\T}$ are antipodal on the circle $\T$.
Since the arc $[-\alpha,\alpha]$ has length $2\alpha=\sqrt2\,\pi>\pi$, every
antipodal pair intersects this arc.  Hence
$
\fm_2(\theta)\ge \frac{\pi}{2\alpha}>0
$
for almost every 
$\theta\in\T$.  Therefore the even and odd decimated sequences
are non-deterministic
by Kolmogorov's prediction formula \eqref{def_nu2}.

\end{example}

 \subsection{Proof of Theorem \ref{thm_FGN}} \label{SEC_34}

Let
\[
\wt Y_n(t):=\frac{Y_n(t)}{\sqrt{\log n}}
=\frac1{\sqrt{n\log n}}\sum_{k=1}^{\lfloor nt\rfloor}\varphi(X_k),
\quad t\in[0,1].
\]
The convergence of the finite-dimensional distributions is 
a particular case of Theorem 1' in the original paper \cite{BM83}.
See also the first assertion of \cite[Theorem~5.1]{NN20}, 
whose proof only uses the assumption $\varphi\in L^2(\g)$.  
When $d=1$, we have $H=1/2$, so $\{X_k\}$ is an i.i.d. standard Gaussian sequence.  
The usual Donsker theorem gives tightness of $\{Y_n\}_{n\geq 2}$,
and therefore $\wt Y_n=Y_n/\sqrt{\log n}$ converges to the zero process.  
We henceforth assume $d\ge2$.

With the Hermite expansion \eqref{her_varphi}, the limiting variance of $\wt Y_n(1)$ is given 
by \eqref{sstar},
because for $d\geq 2$,
\begin{center}
$\rho_H(k)\sim \frac{(2d-1)(d-1)}{2d^2}\, |k|^{-1/d}$
as  $|k|\to\infty$
\end{center}
and only the $d$-th chaos contributes to the logarithmic divergence:
\begin{align*}
 \sum_{|k|\leq n} \rho_H(k)^d \sim 2 \big(\tfrac{(2d-1)(d-1)}{2d^2} \big)^d\log n 
\qquad \text{while for $q > d$,}
\qquad \sum_{k\in\Z} |\rho_H(k)|^q  < +\infty.
\end{align*}
It remains to prove tightness.

Let us first mention that
\begin{align}\label{anyH}
\text{the fractional Gaussian noise is non-deterministic for any $H\in(0,1)$.}
\end{align}
Indeed,   its spectral measure  is absolutely continuous 
with spectral density 
\begin{equation}\label{fmH}
\fm_H(\lambda)
=C_H(1-\cos\lambda)\sum_{\ell\in\Z}|\lambda+2\pi\ell|^{-1-2H},
\quad \lambda\in\T
\end{equation}
for a positive constant $C_H$;
see, e.g.,  \cite[(2.17)]{Ber94}.
This density is positive almost everywhere and satisfies the 
following conditions:
\begin{itemize}
\item by \eqref{fmH}, we have $\fm_H(\lambda) \geq C_H (1 - \cos\theta) \pi^{-1-2H}$
for any $\lambda\in\T$ with $|\lambda| \geq \theta > 0$;
\item for some positive constant $C'_H$, we have $\fm_H(\lambda) \sim C'_H  |\lambda|^{1-2H}$ as $\lambda\to 0$;
see \cite[(2.18)]{Ber94}.
\end{itemize}
Then, it is easy to verify that
 $\log \fm_H\in L^1(\T)$,
 and  
 Kolmogorov's prediction formula \eqref{def_nu2} therefore 
 gives a strictly positive 
one-step innovation variance $\nu^2 > 0$.

Then, the rest of the proof is a simple modification of the proof of Theorem \ref{thm_BMD1}.
Let
\begin{center}
$\F_k:=\sigma\{X_j:j\le k\}$, 
$\Delta_k:=\E[\varphi(X_k)\mid \F_{k-1}]$,
and
$
\xi_k:=\varphi(X_k)-\Delta_k.
$
\end{center}
Then, we have the decomposition
$
\wt Y_n(t)=\wt M_n(t) + \wt N_n(t),
 $
where
\[
\wt M_n(t)=\frac1{\sqrt{n\log n}}\sum_{k=1}^{\lfloor nt\rfloor}\xi_k,
\qquad
\wt N_n(t)=\frac1{\sqrt{n\log n}}\sum_{k=1}^{\lfloor nt\rfloor}\Delta_k.
\]
Since $\rho_H(k)\to0$ as $|k|\to+\infty$, 
the Gaussian sequence $\{X_k\}$ is ergodic.  
Thus, Proposition~\ref{MCLT} implies that
$
\big\{  \frac1{\sqrt n}\sum_{k=1}^{\lfloor nt\rfloor}\xi_k:t\in[0,1] \big\}_{n\ge1}
$
is tight in $\mathbf{D}([0,1])$.  
Therefore, $\wt M_n$ converges in probability to the zero process
in $\mathbf{D}([0,1])$.

\medskip

It remains to prove the  tightness of $\{\wt N_n\}_{n\geq 1}$.
Put
$R_k:=\E[X_k | \F_{k-1}]$,
$\eps_k:=X_k-R_k$,
and
$Z_k:=\frac{R_k}{\sqrt{1-\nu^2}}$.
As in Lemma~\ref{lem_1}, we can write 
\begin{center}
$\Delta_k=P_\tau\varphi(Z_k)$
with  $e^{-2\tau}=1-\nu^2 < 1$.
\end{center}
Here and below, $P_\tau$ denotes the Ornstein-Uhlenbeck semigroup associated
with the Gaussian space generated by the sequence $\{X_k\}_{k\in\Z}$.
Let $\rho_Z(h)=\E[Z_0Z_h]$.  
As one can see from \eqref{seqa}-\eqref{seqb} in the proof of Lemma~\ref{lem_1},
we can find  a sequence $(a_h)_{h\in\Z}\in\ell^2(\Z)$ such that
\[
(1-\nu^2)\rho_Z(h)=\rho_H(h)-a_h.
\]

\noi
 Since $d\ge2$ and $|\rho_H(h)|\le C(1+|h|)^{-1/d}$, 
 we get $(a_h)_{h\in\Z}\in\ell^d(\Z)$ and thus\footnote{The constant 
 $C$ may vary from line to line 
 in this proof and $C$ does not depend on $n$, $t$, nor $s$.}
 
 \noi
\begin{align}\label{rhoZ}
\sum_{|h|\le m}|\rho_Z(h)|^d\leq C\log(e+m)
\end{align}

\noi
for any $m\geq 1$.  Next, we can carry the same computations
from \eqref{mark1} to \eqref{mark2} with 
$p = 1 +e^{2\tau} > 2$:
for $0\le s\le t\le1$,
\begin{align*}
\| \wt N_n(t) - \wt N_n(s) \|_p^2
&\leq  \frac{    \lfloor nt \rfloor  -  \lfloor ns \rfloor     }{n \log n} 
\sum_{q\geq d} c_q^2 q! \sum_{| k| \leq  \lfloor nt \rfloor  -  \lfloor ns \rfloor    } | \rho_Z(k)|^q 
\leq C \frac{\lfloor nt\rfloor-\lfloor ns\rfloor}{n},
\end{align*}

\noi
where the last inequality is obtained by applying \eqref{rhoZ} with $|\rho_Z|\leq1$ 
and $\| \varphi\|^2_{L^2(\g)}  = \sum_{q\geq d} c_q^2 q! <\infty$.
Consequently, 
Lemma~\ref{lem_NN20} gives us the tightness of $\{\wt N_n\}_{n\geq 2}$,
and thus we have the tightness of $\{ \wt{Y}_n\}_{n\geq 2}$.
Hence, $\wt Y_n \LRA \s_\ast W$  in $\mathbf{D}([0,1])$.

Finally,  the desired weak convergence of  the linearly interpolated version of $\wt Y_n$
 in $\mathbf{C}([0,1])$
follows from Proposition \ref{prop_C} with $b_n = \sqrt{\log n}$.
\hfill $\square$

\end{document}